\documentclass[10pt, reqno]{amsart}

\usepackage{amsmath,amssymb,amsthm}

\usepackage{tikz}

\usepackage{caption}
\usepackage{graphicx}
\usepackage{graphics}
\usepackage{algorithm}
\usepackage{algorithmicx}
\usepackage{algpseudocode}

\newtheorem{theorem}{Theorem}

\newtheorem{lemma}{Lemma}

\begin{document}

\title[]{Many critical points for \\
discrete Riesz energy on \Large $\mathbb{T}^2$}

\author[]{Fran\c{c}ois Cl\'ement \and Stefan Steinerberger}

\address{Department of Mathematics, University of Washington, Seattle}
 \email{fclement@uw.edu }

 \address{Department of Mathematics and Department of Applied Mathematics, University of Washington, Seattle}
 \email{steinerb@uw.edu}

\begin{abstract} 
It is widely believed that the energy functional $E_p:(\mathbb{S}^2)^n \rightarrow \mathbb{R}$  
$$ E_p = \sum_{i,j=1 \atop i \neq j}^{n} \frac{1}{\|x_i-x_j\|^p}$$
has a number of critical points, $\nabla E(x) = 0$, that grows exponentially in $n$. Despite having been extensively tested and being physically well motivated, no rigorous result in this direction exists. We prove a version of this result on the two-dimensional flat torus $\mathbb{T}^2$ and show that there are infinitely many $n \in \mathbb{N}$ such that the number of critical points of $E_p: (\mathbb{T}^2)^n \rightarrow \mathbb{R}$ is at least $\exp(c \sqrt{n})$ provided $p \geq 5 \log{n}$. We also investigate the special cases $n=3,4,5$ which turn out to be surprisingly interesting.
\end{abstract}

\maketitle

\section{Introduction and Results}
\subsection{The problem}

We consider the classical Riesz energy minimization problem for a finite number of points. It is frequently studied on the sphere: given $p>0$ find $x_1, \dots, x_n \in \mathbb{S}^2$ so as to minimize 
$$ E(x_1, x_2, \dots, x_n) = \sum_{i,j=1 \atop i \neq j}^{n} \frac{1}{\|x_i - x_j\|^p}.$$
This problem is classical, known as the Thomson problem \cite{thomson} when $p=1$, and has been studied for a very long time, we refer to \cite{boro,brauchart, hardin, saff}. We were motivated by a very simple question.
\begin{quote}
    How many critical points does the energy $E(x_1, \dots, x_n)$ have?
\end{quote}
There are two symmetries: relabeling the points (giving rise to $n!$ critical points from a single critical point) and isometries on the sphere. We interpret the question after these two symmetries have been factored out (sometimes known as `geometrically distinct configurations').
Erber and Hockney \cite{erber, erber2} undertook numerical simulations (for $p=1$) suggesting that the number of critical points grows exponentially in $n$: they estimate the number to be $\sim 0.382 \exp (0.0497 n)$. Exponential growth has also been suggested by other authors \cite{am2, bendito, calef, mehta}. Amore-Figueroa-Diaz-Lopez-Vincent \cite{amore} suggest $\sim 0.00017 \exp(0.12 n)$.
Exponential growth is sometimes taken for granted in the physics literature: it \textit{seems difficult to avoid the same conclusion that the total number of distinguishable packings rises
exponentially with increasing N at fixed density} \cite{still1}, this \textit{rests on the expectation that particle packings can be rearranged essentially independently in the two halves of a large system} \cite{still2}. A proof based \textit{on physical grounds} was given by Stillinger \cite{still3}. We do not even know of an argument showing that there are at least, say, $n^{10}$ critical points.

\subsection{Critical Points on the Torus}
The main purpose of this paper is to study the corresponding problem on the flat torus $\mathbb{T}^2 \cong [0,1]^2$ and to establish, among other things, some nontrivial lower bounds on the number of critical points.

\begin{theorem}
    There exist constants $c_1, c_2 > 0$ and an infinite set of integers $A \subset \mathbb{N}$ such that if $n \in A$ and $p > 5 \log{n}$, then the function
    $E: (\mathbb{T}^2)^n \rightarrow \mathbb{R}$ 
    $$ E(x_1, x_2, \dots, x_n) = \sum_{i,j=1 \atop i \neq j}^{n} \frac{1}{\|x_i - x_j\|^p}$$
    has at least $c_1 \exp(c_2 \sqrt{n})$ critical points.
\end{theorem}

\textbf{Remarks.} Several remarks are in order.
\begin{enumerate}
\item We use $\| x_i - x_j \|$ to denote the geodesic distance $d(x_i, x_j)$ on $\mathbb{T}^2 \cong [0,1]^2$. 
\item Just as on $\mathbb{S}^2$, the number of critical points refers to the number after all the symmetries have been factored out. 
    \item The subset of integers $A \subset \mathbb{N}$ is $A = \left\{60 \ell^2: \ell \in \mathbb{N} \right\}.$ The proof could be modified to cover larger sets of integers, we chose this particular subset for ease of exposition. The proof also gives explicit values for the constants $c_1$ and $c_2$, these have not been optimized.
    \item The condition $p > 5 \log{n}$ is required for the proof. There is no reason to assume that it is necessary for the statement to be true. Our construction might even be able to exhibit this number of such critical points for smaller values of $p$, however, some new ideas are required to make this rigorous; we comment on this after the proof.
\end{enumerate}

When $p$ is `large', most of the energy comes from the nearest-neighbor interactions: in that case the global geometry of $\mathbb{T}^2$ should \textit{not} come into play, everything should behave more or less as it does on $\mathbb{S}^2$ (or, for that matter, on any other compact two-dimensional Riemannian manifold without boundary). In particular, one would be inclined to expect an exponential number of critical points.

\begin{center}
\begin{figure}[h!]
\begin{tikzpicture}
\node at (0,0) { \includegraphics[width=0.3\textwidth]{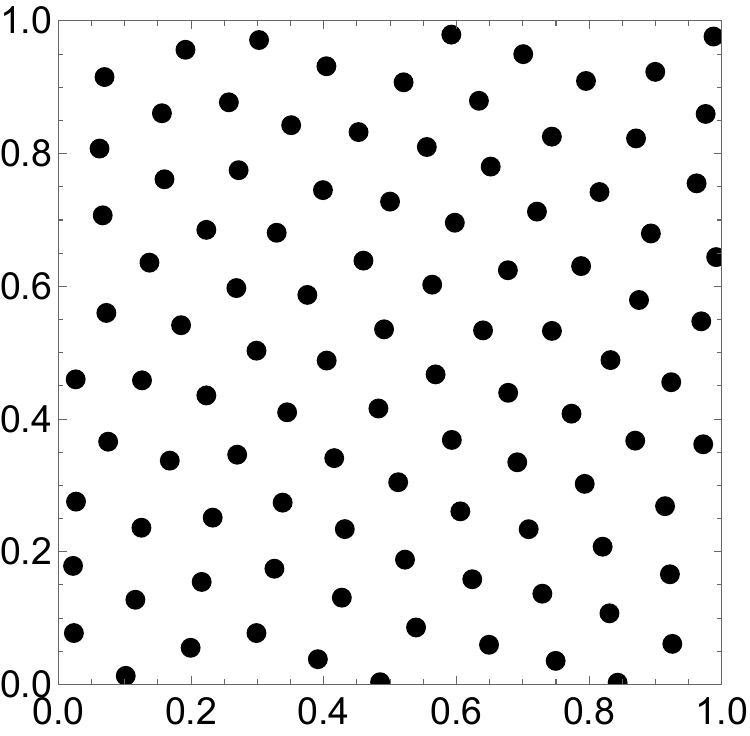}};
\node at (4,0) { \includegraphics[width=0.3\textwidth]{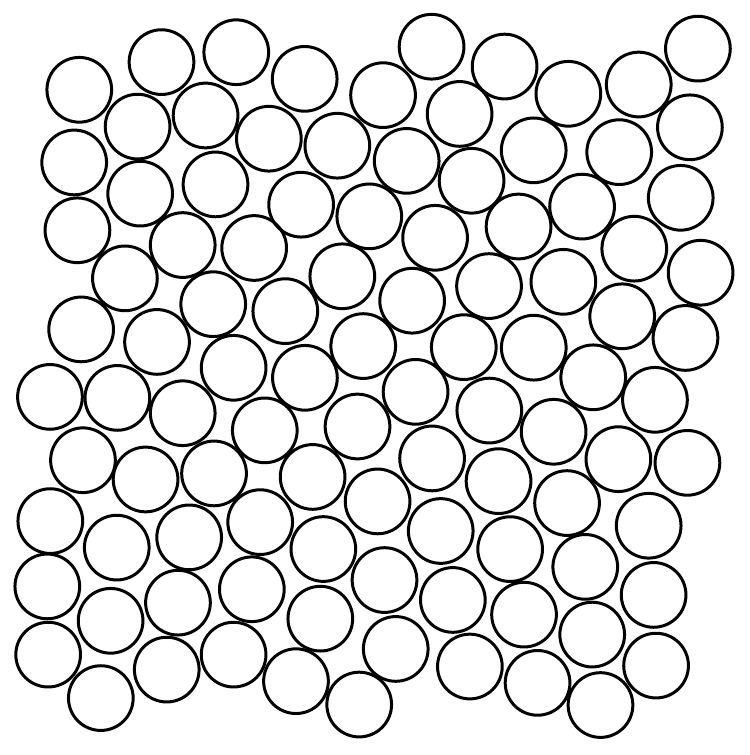}};
\node at (8,0) { \includegraphics[width=0.3\textwidth]{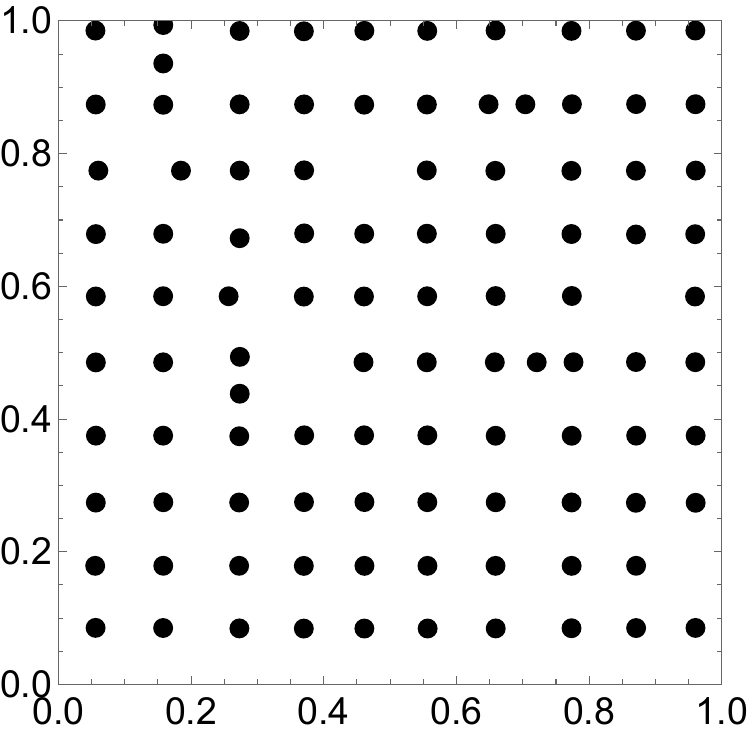}};
\node at (2.2,-2.4) {$\sum_{i \neq j} \|x_i - x_j\|^{-4}$};
\node at (8,-2.4) {$\sum_{i \neq j} \log\left(1/\|x_i - x_j\|\right)$};
    \end{tikzpicture}
    \caption{Critical points found via gradient descent starting with 100 random points. Left: $p=4$. Middle: the case $p=4$ with circles drawn around each point.  Right: logarithmic energy $p=0$.}
    \label{fig:surprise}
    \end{figure}
\end{center}

\subsection{A surprise for small $p$.}
Somewhat to our surprise, see Fig. \ref{fig:surprise} (right) and Fig. \ref{fig:intricate}, the logarithmic energy (also sometimes known as $0-$Riesz energy because it can be recovered from a suitable renormalization of the energy $E_p$ as $p \rightarrow 0^+$) behaves differently; it is defined as 
$$ E_0(x_1, x_2, \dots, x_n) := \sum_{i,j=1 \atop i \neq j}^{n} \log\left( \frac{1}{\|x_i - x_j\|} \right).$$
Even when starting with a set of completely random points, the gradient flow appears to transport the points to a critical point of the energy functional where the points appear to be arranged in a `grid-like' fashion (certain lines contain many points). This suggests a much more intricate global structure. We note that if there are $\leq c\sqrt{n}$ lines containing somewhere between $0$ and $c \sqrt{n}$ points, then the number of such configurations is $\leq (c \sqrt{n})^{c \sqrt{n}} \sim \exp(c (\log{n}) \sqrt{n})$ which would be quite a bit smaller than what is expected on $\mathbb{S}^2$.
We believe this to be a fascinating phenomenon. There are at least two very natural questions.
\begin{enumerate}
    \item  Are `most' critical points of the logarithmic energy on $\mathbb{T}^2$ `grid-like'? 
    \item Does the gradient flow initialized with $n$ random points lead to a critical point that is arranged in a `grid-like' fashion with high likelihood?
\end{enumerate}
`Grid-like' and `most' are intentionally kept vague. The critical points found by gradient flow turn out to have a most intricate structure, see Fig. \ref{fig:intricate}.
We also remark that it is conceivable that a great many `non-grid-like' critical points may exist, perhaps they are very difficult to find: maybe they have a small basin of attraction in the energy landscape and gradient descent is unlikely to reveal them (cf. the `unseen species problem' \cite{amore}).

\begin{center}
    \begin{figure}[h!]
      \begin{tikzpicture}
      \node at (0,0) {\includegraphics[width=0.3\textwidth]{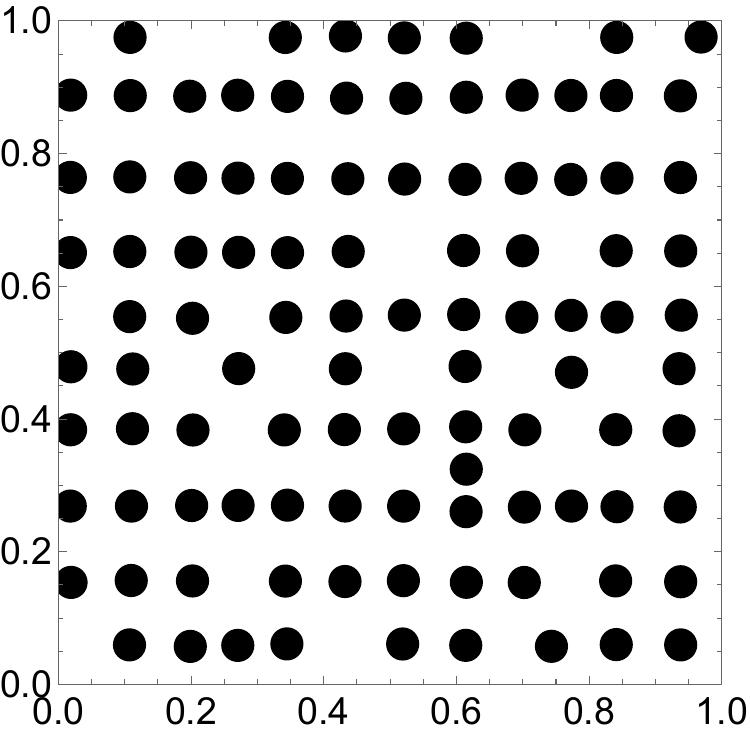}};    
       \node at (4,0) {\includegraphics[width=0.3\textwidth]{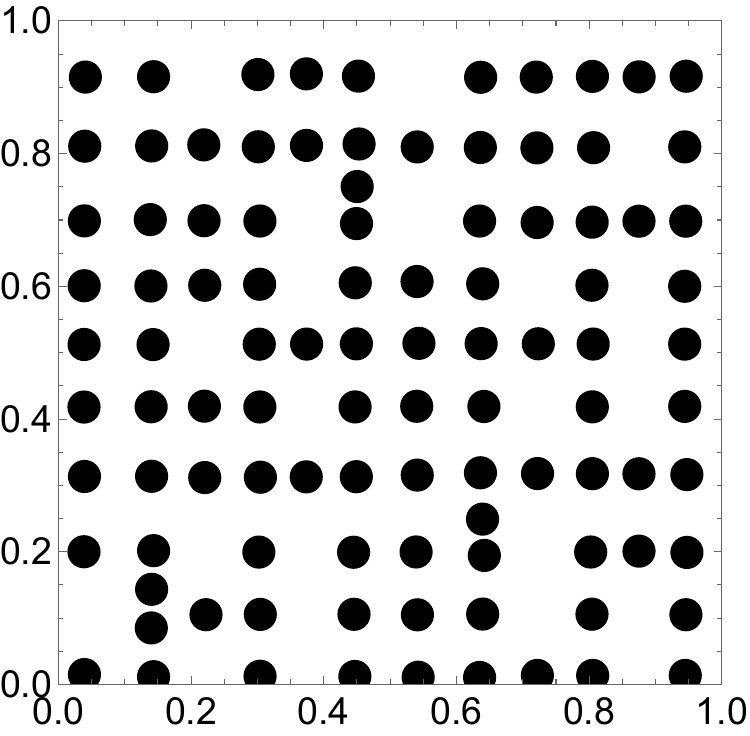}}; 
       \node at (8,0) {\includegraphics[width=0.3\textwidth]{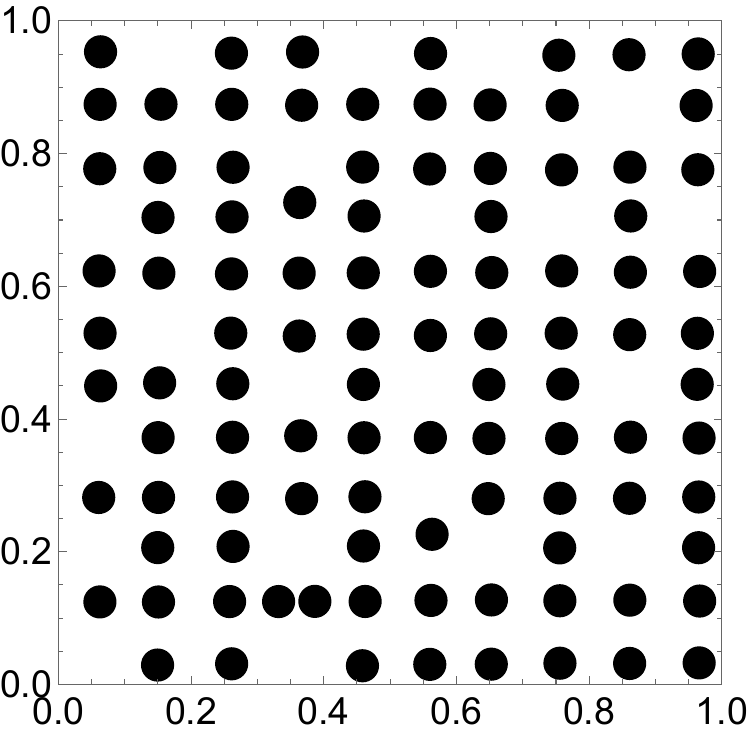}}; 
      \end{tikzpicture}
        \caption{Critical points for the logarithmic energy obtained via gradient descent from three sets of 100 random initial points.}
        \label{fig:intricate}
    \end{figure}
\end{center}

\vspace{-20pt}

\subsection{Related results}
The Riesz energy of points on $\mathbb{T}^2$ has been considered before \cite{bence, hardin2, marzo}. However, we are not aware of any directly related results, in particular on the number of critical points (either on $\mathbb{S}^2$ or on $\mathbb{T}^2$). The problem of minimizing the energy $\mathbb{S}^2$, especially when $p=1$, is well-known and often used as a benchmark problem for optimization algorithms \cite{alt, morris, xiang}. We were at least partially motivated by recent work of Nagel \cite{nagel} and Bilyk-Nagel-Ruohoniemi \cite{bilyk} concerning the Fibonacci set on $\mathbb{T}^2$ (the 5-element Fibonacci set also arises as the minimizer for our problem when $n=5$, see \S \ref{sec:fibo}). We also note the connection to the crystallization conjecture \cite{blanc} and the conjecture of Cohn--Kumar \cite{cohnk}.

\section{Small number of points and special configurations}
The problem on the flat torus $\mathbb{T}^2$ is quite fascinating. Even $n=3,4,5$ points on $\mathbb{T}^2$ already give rise to interesting structures that are very different from $\mathbb{S}^2$. 

\subsection{The case $n=3$.}
There are at least three configurations of interest. These are shown in Fig. \ref{fig:n3} and are (up to all the symmetries) 
\begin{align*}
  T_1 &= \left\{ (0,0), (\tfrac12, 0), (\tfrac12, \tfrac12) \right\} \\
  T_2 &= \left\{ (0,\tfrac12), (\tfrac12, \tfrac14), (\tfrac12, \tfrac34) \right\} \\
 T_3 &= \left\{ (0,0), (\tfrac12, \tfrac{2- \sqrt{3}}{2}), (\tfrac{2- \sqrt{3}}{2}, \tfrac12) \right\}. 
\end{align*}

\begin{center}
    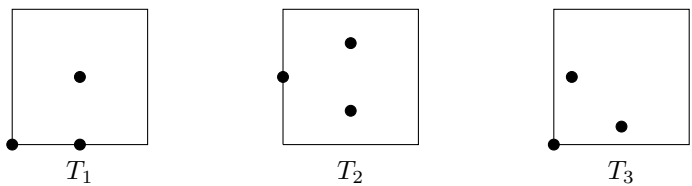
\begin{figure}[h!]
     \begin{tikzpicture}[scale=0.9]
         \draw (0,0) -- (2,0) -- (2,2) -- (0,2) -- (0,0);
       \draw (4,0) -- (6,0) -- (6,2) -- (4,2) -- (4,0);
         \draw (8,0) -- (10,0) -- (10,2) -- (8,2) -- (8,0);
    \filldraw (0,0) circle (0.08cm);
   \filldraw (1,0) circle (0.08cm);
    \filldraw (1,1) circle (0.08cm);
    \filldraw (4,1) circle (0.08cm);
        \filldraw (5,0.5) circle (0.08cm);
        \filldraw (5,1.5) circle (0.08cm);
      \filldraw (8,0) circle (0.08cm);
      \filldraw (9,0.266) circle (0.08cm);
    \filldraw (8.266, 1) circle (0.08cm);
    \node at (1, -0.4) {$T_1$};
    \node at (5, -0.4) {$T_2$};
    \node at (9, -0.4) {$T_3$};
     \end{tikzpicture}
        \caption{Three interesting sets of three points.}
        \label{fig:n3}
    \end{figure}
\end{center}
\vspace{-5pt}

An explicit computation shows that $T_1$ has the smallest energy among the three for $0 \leq p \leq 4.505$. $T_2$ minimizes the energy among the three for $4.505 \leq p \leq 29.65$. Finally, once $p \geq 29.65$, the set $T_3$ has the smallest energy among the three. It seems reasonable to conjecture that $T_1$ and $T_2$ are optimal for `small' and `intermediate' value of $p$, respectively. 

\begin{center}
\begin{figure}[h!]
    \begin{tikzpicture}
    \draw [thick, ->] (0,0) -- (9,0);
    \draw [thick] (0,-0.1) -- (0,0.1);
    \node at (0, -0.3) {$p=0$};
      \draw [thick] (2,-0.1) -- (2,0.1); 
          \node at (2, -0.3) {$p=4.505$};
     \draw [thick] (6,-0.1) -- (6,0.1);  
         \node at (6, -0.3) {$p=29.653$};
         \node at (1, 0.2) {$T_1$};
       \node at (3.5, 0.2) {$T_2$};
        \node at (7.8, 0.2) {close to $T_3$};
\end{tikzpicture}
\vspace{-5pt}
\caption{Conjectured minimal energy configuration.}
\end{figure}
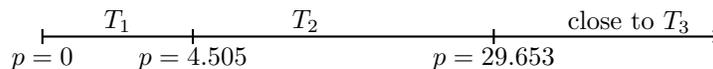  
\end{center}
\vspace{-5pt}
Even though $T_3$ is the set of point maximizing the minimal distance between any pair of points and thus the formal limit as $p \rightarrow$, $T_3$ is \textit{not} optimal for any large but finite value of $p$: for large values of $p$ any energy minimizing configuration has to be close to $T_3$ but is distinct from $T_3$.
 
\begin{theorem}
 As $p \rightarrow \infty$, every global minimizer of
 $$\min_{x_1, x_2, x_3 \in \mathbb{T}^2} \quad \frac{1}{\|x_1 - x_2\|^p} +  \frac{1}{\|x_1 - x_3\|^p} +  \frac{1}{\|x_2 - x_3\|^p} $$
 has to be close to $T_3$ but is distinct from it. Each such configuration can only be optimal for a finite interval of $p$: as $p$ varies, there are infinitely many distinct global minimizers (even after factoring out the symmetries).
\end{theorem}

Phase transitions themselves are not uncommon for these types of problems. Melnyk-Knop-Smith \cite{melnyk} remarked in 1977 that there appear to be two optimal solutions for $5$ points on $\mathbb{S}^2$, one that is optimal for $p \leq 15.04\dots$ and the other for larger values of $p$ (a computer-assisted proof was recently given by Schwartz \cite{schw2}). It appears as if $n=3$ points on $\mathbb{T}^2$ appear to give rise to an infinite number of phase transitions which, to the best of our knowledge, has not yet been observed in any other setting. It also means that a complete characterization of the minimizers for all $p$ might be difficult; on the other hand, for small values of $p$, the phase transitions might behave as they do on $\mathbb{S}^2$ with minimizers on either side of the phase transition that remain global minimizers for an entire range of $p$.

\subsection{The case $n=4$.} The case of $4$ points has a similar degree of complexity. For small values of $p$ the natural candidate extremizer are 4 points arranged in a square of side-length $1/2$. 
For large $p$, the extremal configuration is close to the configuration that maximizes the minimal distance between any two points. There is a one-parameter family of sets $S_{\alpha}$ interpolating between these two examples
$$ S_{\alpha} = \left\{ (0,0), \left( \tfrac{\tan{\alpha}}{2} , \tfrac12 \right), \left( \tfrac{1}{2}, \tfrac{ \tan{\alpha}}{2} \right), \left( \tfrac{1 + \tan{\alpha}}{2} , \tfrac{1 + \tan{\alpha}}{2} \right) \right\} \qquad 0 \leq \alpha \leq \tfrac{\pi}{12}.$$

$S_{0}$ corresponds to the square, $S_{\pi/12}$ is the best packing of 4 equal circles in $\mathbb{T}^2$. The one-parameter family has the property that the minimal (toroidal) distance between any pair of points is strictly increasing along the flow; we found the family discussed in Dickinson-Guillot-Keaton-Xhumari \cite[Figure 6]{dick}.

\begin{center}
    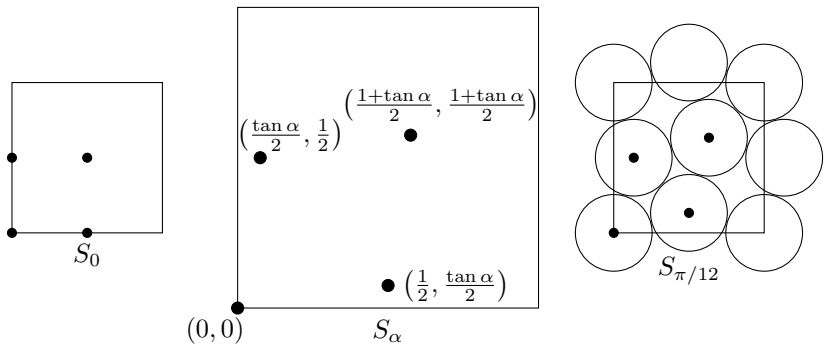
\begin{figure}[h!]
     \begin{tikzpicture}
         \draw (0,0) -- (2,0) -- (2,2) -- (0,2) -- (0,0);
       \draw (3,-1) -- (7,-1) -- (7,3) -- (3,3) -- (3,-1);
         \draw (8,0) -- (10,0) -- (10,2) -- (8,2) -- (8,0);
    \filldraw (0,0) circle (0.06cm);
   \filldraw (1,0) circle (0.06cm);
    \filldraw (1,1) circle (0.06cm);
   \filldraw (0,1) circle (0.06cm);
       \node at (2.7, -1.3) {$(0,0)$};
    \filldraw (3,-1) circle (0.08cm);
        \filldraw (5, -1 +0.3) circle (0.08cm);
        \filldraw (3 +0.3,1) circle (0.08cm);
       \filldraw (3 +2.3,-1 + 2.3) circle (0.08cm); 
      \node at (5.7, 1.7) {$\left( \tfrac{1 + \tan{\alpha}}{2}, \tfrac{1 +\tan{\alpha}}{2} \right)$};
       \node at (5.9, -0.7) {$\left( \tfrac{1}{2}, \tfrac{\tan{\alpha}}{2} \right)$};
        \node at (3.7, 1.3) {$\left(\tfrac{\tan{\alpha}}{2}, \tfrac{1}{2} \right)$};
      \filldraw (8,0) circle (0.06cm);
      \filldraw (9,0.266) circle (0.06cm);
    \filldraw (8.266, 1) circle (0.06cm);
     \filldraw (9.266, 1.266) circle (0.06cm);
     \draw (8,0) circle (0.51cm);
     \draw (10,2) circle (0.51cm);
          \draw (8,2) circle (0.51cm);
            \draw (10,0) circle (0.51cm);   
     \draw (9,0.2660) circle (0.51cm); 
       \draw (8.266,1) circle (0.51cm);  
     \draw (9.266, 1.2660) circle (0.51cm);     
      \draw (10.266, 1) circle (0.51cm);  
            \draw (9, 2.266) circle (0.51cm);
    \node at (1, -0.3) {$S_{0}$};
    \node at (5, -1.3) {$S_{\alpha}$};
    \node at (9, -0.5) {$S_{\pi/12}$};
     \end{tikzpicture}
        \caption{A one-parameter family, $0 \leq \alpha \leq \pi/12$.}
        \label{fig:n4}
    \end{figure}
\end{center}

The set $S_{\alpha}$ has a Riesz energy that is easy to compute: the distance between $(s_1, s_2), (s_1, s_3), (s_2, s_4)$ and $(s_3, s_4)$ is $\sec{(\alpha})/2$. The distance between $(s_1, s_4)$ and $(s_2,s_3)$ is $(1-\tan{\alpha})/\sqrt{2}$. Therefore,
$$ \sum_{i,j=1 \atop i < j}^{4} \frac{1}{\|x_i - x_j\|^p} = 4 \left( \frac{2}{\sec{\alpha}} \right)^p + 2 \left(\frac{\sqrt{2}}{1 - \tan{\alpha}}\right)^p.$$
Another natural candidate set is the square with two of the points offset by $1/4$, i.e.
$ S_{\mbox{\tiny shift}} = \left\{ (0,0), (\tfrac14, \tfrac12), (\tfrac12, 0), (\tfrac34, \tfrac12)\right\}$
with energy
$$ \sum_{i,j=1 \atop i < j}^{4} \frac{1}{\|x_i - x_j\|^p} = 2 \cdot 2^p + 4 \left( \frac{4}{\sqrt{5}} \right)^p.$$
These examples imply the following observations.
\begin{enumerate}
    \item  For $0 < p < 4.506...$ the square lattice $S_0$ has a lower energy than $S_{\mbox{\tiny shift}}$ and any other set $S_{\alpha}$ from the one-parameter family.
    \item For $4.506... \leq p \leq 26.3...$ the set $S_{\mbox{\tiny shift}}$ has a smaller energy than all $S_{\alpha}$. 
    \item For $p > 26.3...$ there exists an $0 < \alpha < \pi/12$ such that $S_{\alpha}$ has the smallest energy among all $S_{\alpha}$ (which is also smaller than the energy of $S_{\mbox{\tiny shift}}$). 
\end{enumerate}

 However, just as in the case of $n=3$, when $n=4$ then large values of $p$ admit an infinite number of different global minimizers.
 
\begin{theorem}
Every global minimizer of
 $$ \min_{x_1, x_2, x_3, x_4 \in \mathbb{R}^2} \quad \sum_{i,j=1 \atop i < j}^{4} \frac{1}{\|x_i - x_j\|^p}$$
  can only be a minimizer for a finite interval of $p$. As $p$ varies, there are infinitely many distinct global minimizers (even after factoring out the symmetries).
\end{theorem}

It is not surprising that, as $p \rightarrow \infty$, the global minimizers get closer and closer to $S_{\pi/12}$ however, they are always distinct from $S_{\pi/12}$. Our argument shows that there exist at least a countable number of minimizers; it seems reasonable to assume that there are uncountably many and that they vary continuously with $p$. As (weakly) supporting evidence we note that the $\alpha$ attaining the minimal energy within the family $S_{\alpha}$ appears to behave approximately like $\alpha(p) = \pi/12 - c/p$ for some $c \sim 0.6$.

\subsection{The case $n=5$.} \label{sec:fibo} The case of $n=5$ points is highly nontrivial even on the sphere $\mathbb{S}^2$ and was solved by Schwartz \cite{schwartz, schw2} with a computer-assisted proof; no `simple' proof is known. 

\begin{center}
    \begin{figure}[h!]
  \begin{tikzpicture}
      \node at (0,0) {\includegraphics[width =0.3\textwidth]{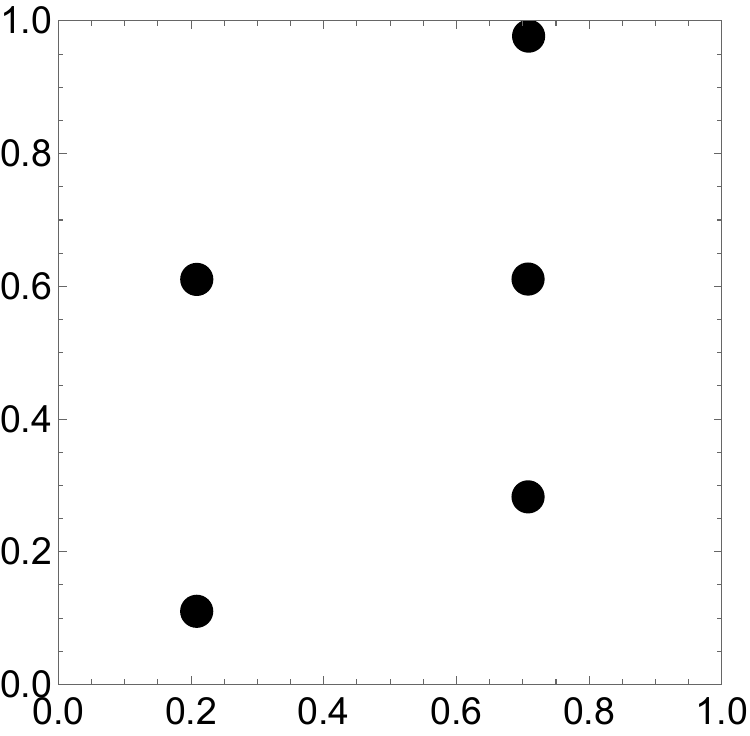}};
       \node at (4,0) {\includegraphics[width =0.3\textwidth]{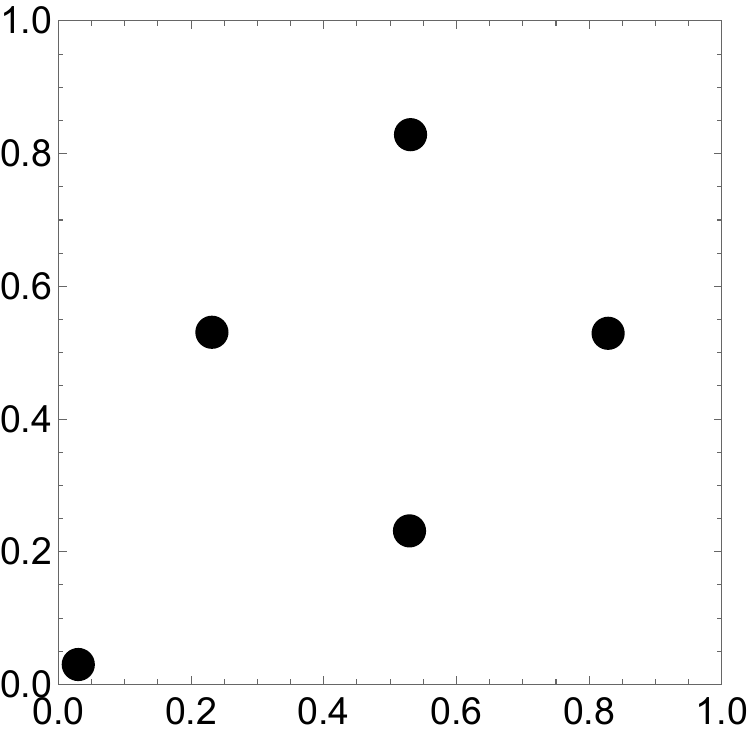}};
       \node at (8,0) {\includegraphics[width =0.3\textwidth]{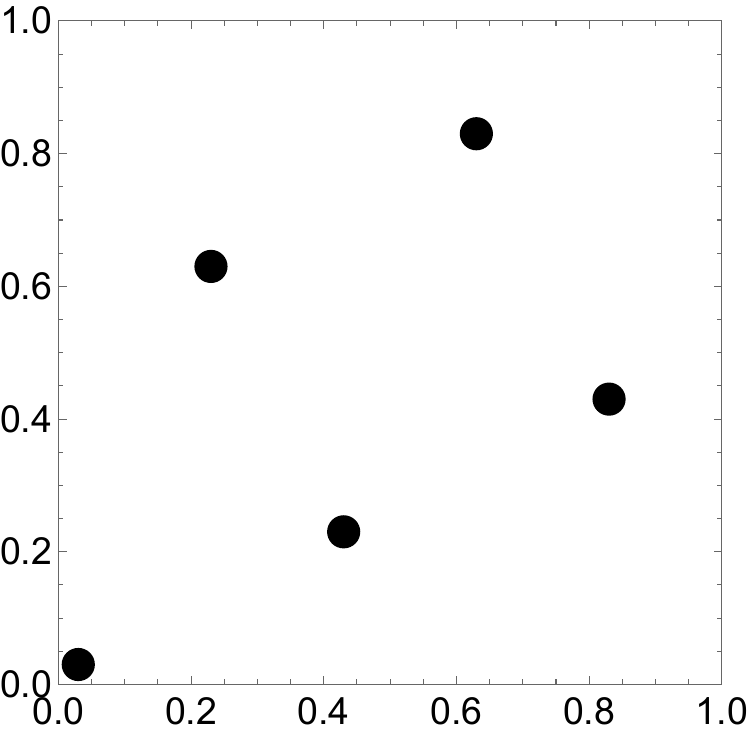}};
  \end{tikzpicture}
        \caption{Candidates for optimal configurations when $n=5$. }
        \label{fig:n5}
    \end{figure}
\end{center}

The minimal energy configuration has to converge, as $p \rightarrow \infty$, to the solution of the disk packing problem. This solution is given by the Fibonacci set
$$ F_5 = \left\{ (\tfrac{0}{5},\tfrac{0}{5}), \left( \tfrac{1}{5}, \tfrac35\right),  \left( \tfrac{2}{5}, \tfrac15\right), \left( \tfrac{3}{5}, \tfrac45\right),  \left( \tfrac{4}{5}, \tfrac45\right)\right\}.$$
It was conjectured by Dickinson (1994, \cite{dick2}) and Melissen (1997, \cite{mel}) that for any set of five points one has
$$ \max_{x_1, \dots, x_5 \in \mathbb{T}^2} ~\min_{i \neq j} \|x_i - x_j\| \leq \frac{1}{\sqrt{5}}$$
with equality if and only if the five points are $F_5$ (up to symmetries). This conjecture was proven by Dickinson-Guillot-Keaton-Xhumari \cite{dick}.
However, in contrast to the two other cases $n=3,4$, we can show that the minimal energy configuration for all $p$ sufficiently large is \textit{exactly} $F_5$ (up to symmetries).

\begin{theorem} \label{prop:5} There exists $p_0 > 0$ such that for all $p \geq p_0$, the minimal energy configuration is given by the Fibonacci set $F_5$: for any $x_1, \dots, x_5 \in \mathbb{T}^2$
$$ \sum_{i,j=1 \atop i \neq j}^{5} \frac{1}{\|x_i - x_j\|^p} \geq 20 \cdot 5^{p/2}.$$
\end{theorem}

Our proof does not give any information about $p_0$ (basic numerical experiments show that $p_0 \geq 4.1$). Theorem 4 uses the result of Dickinson-Guillot-Keaton-Xhumari \cite{dick}. To make the argument quantitative and get an explicit bound on $p_0$, one would need (at least following our line of reasoning) a stability version that describes the `second-best' solution to the packing problem: if $x_1, \dots, x_5 \in \mathbb{T}^2$ are not close to $F_5$ (in whatever metric one desires), then 
$$ \min_{i \neq j} \|x_i - x_j\| \leq \frac{1}{\sqrt{5}} - \varepsilon_0.$$
 The truth of this stability version follows from compactness: if the statement was false, then one could let $\varepsilon \rightarrow 0^+$, obtain a sequence of configurations that are all far from $F_5$, compactness of $(\mathbb{T}^2)^5$ enforces the existence of a convergent subsequence which would lead to a second extremal configuration different from $F_5$ which is a contradiction. However, as is customary for such arguments, invoking compactness makes it impossible to get quantitative control on $\varepsilon_0$.

\subsection{Larger values of $n$.} The best-packing problem on $\mathbb{T}^2$, the limit of our problem as $p \rightarrow \infty$ has been studied for larger values of $n$. Musin-Nikitenko \cite{musin} solve the problem for $n=6,7,8$ and conjecture an optimal configuration for $n=9$. Connelly-Funkhouser-Kuperberg-Solomonides \cite{conn} arrive at the same conjecture for $n=9$ and conjecture optimal configurations for $10 \leq n \leq 16$. 
Variations of the problem have also been studied by Heppes \cite{heppes} and
Lubachevsky-Graham-Stillinger \cite{lub}.

   \begin{center}
   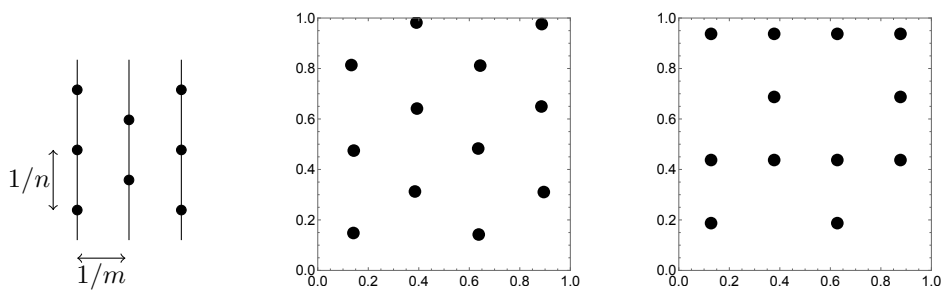
\begin{figure}[h!]
        \begin{tikzpicture}[scale=0.8]
        \node at (6,1.5) {\includegraphics[width=0.3\textwidth]{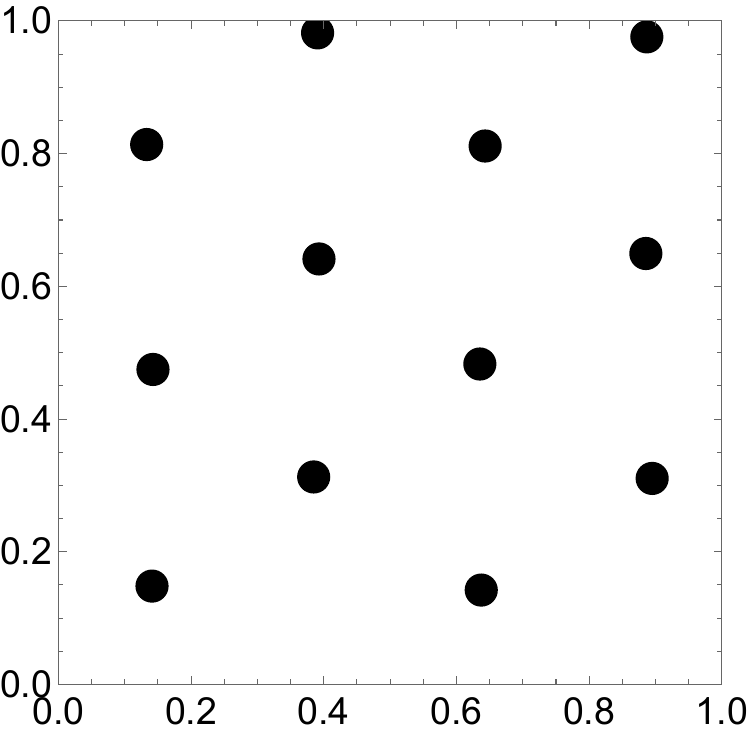}};
            \draw (0,0) -- (0,3);
            \filldraw (0, 0.5) circle (0.08cm);
            \filldraw (0, 2.5) circle (0.08cm);
           \filldraw (0, 1.5) circle (0.08cm);   
             \draw (0.86,0) -- (0.86,3); 
           \filldraw (0.86, 1) circle (0.08cm); 
            \filldraw (0.86, 2) circle (0.08cm);   
           \draw (1.73,0) -- (1.73,3);   
               \filldraw (1.73, 0.5) circle (0.08cm);
            \filldraw (1.73, 2.5) circle (0.08cm);
           \filldraw (1.73, 1.5) circle (0.08cm);          
                 \draw [<->] (0,-0.3) -- (0.8, -0.3);  
           \node at (0.4, -0.6) {$1/m$};
           \draw [<->] (-0.4, 0.5) -- (-0.4, 1.5);
           \node at (-0.8, 1) {$1/n$};
            \node at (12,1.5) {\includegraphics[width=0.3\textwidth]{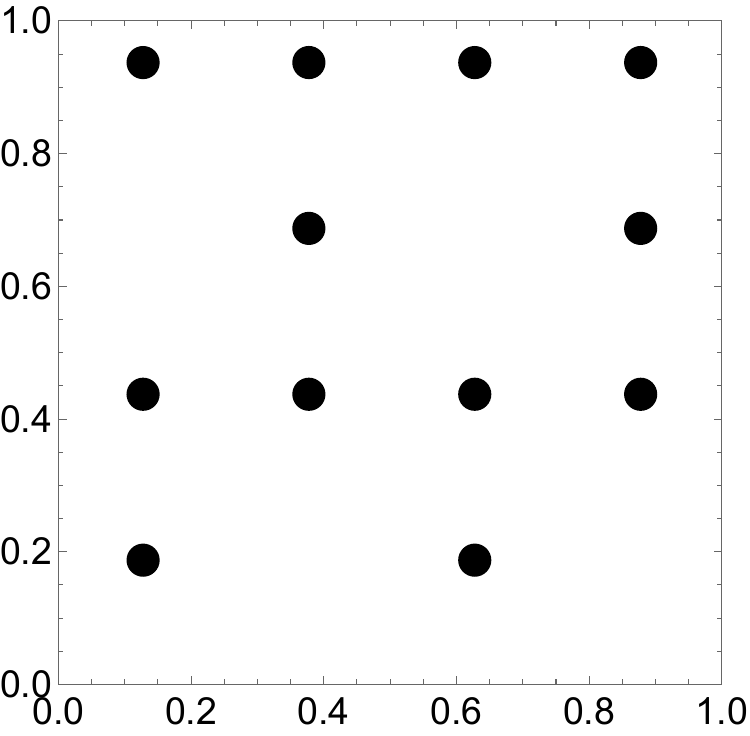}};
           \end{tikzpicture}
        \caption{Left: a local picture of a `Type I Packing' \cite{conn}. Middle: Type I packing for 12 points. Right: a configuration of 12 points with smaller logarithmic energy than the Type I packing.}
        \label{fig:special}
    \end{figure}    
    \end{center}
    
The paper \cite{conn} discusses two special families of configurations (`Type I Packing' and `Type II Packing') which only exist for certain values of $n$ and deserve to be highlighted. These configurations are particularly good when it comes to maximizing the minimal distance between any pair of points; this makes them natural candidates when it comes to minimizing the $p-$energy for $p$ large. Indeed, at least the example corresponding to 12 points (Fig. \ref{fig:special}, middle) is automatically discoverable using gradient descent with a random initial set and $p = 4$. Small values of $p$ can be very different: in the example above, 12 points on $\mathbb{T}^2$ the logarithmic energy is smaller for a grid-like arrangement (Fig. \ref{fig:special}, right).

\section{Proof of Theorem 1}
\begin{proof} The proof of Theorem 1 has three steps.
\begin{enumerate}
    \item We construct, for suitable $n \in \mathbb{N}$, a family of $c_2 \exp(c_3 \sqrt{n})$ many different initial configurations of $n$ points on $\mathbb{T}^2$. Most will not be critical points.
    \item We start (continuous) gradient descent with a particular configuration. The gradient flow converges to a critical point.
    \item The main part of the argument consists in showing that we are able to reconstruct the initial configuration from the critical point recovered by gradient descent. Therefore the number of critical points has to be at least as large as the number of initial configurations and the result follows.
\end{enumerate}

The proof is structured so as to reflect these three steps. There are various times when a constant has to be chosen and we will pick concrete constants for which the statement is true without trying to optimize the constant.\\

\textbf{Step 1.}  Pick an integer $m \in \mathbb{N}$ which is also a multiple of 10 and partition $\mathbb{T}^2$ using $m$ parallel lines of width $1/m$. After that, we place $4m/5$ points equispaced on each of the $m$ lines (for a total of $ n = 4m^2/5$ points placed in $\mathbb{T}^2$). Moreover, we arrange the points in such a way that they are off-set on consecutive lines as shown in Figure \ref{fig:config}. Since $m$ is a multiple of 10 and thus even, this is always possible. 

   \begin{center}
   \begin{figure}[h!]
        \begin{tikzpicture}
            \draw (0-6,0) -- (0-6,3) -- (3-6,3) -- (3-6,0) -- (0-6,0);
            \draw [ultra thick] (0-6,0) -- (0-6,3);
             \draw [ultra thick] (0.5-6,0) -- (0.5-6,3);   
             \draw [ultra thick] (1-6,0) -- (1-6,3);
             \draw [ultra thick] (1.5-6,0) -- (1.5-6,3);   
                \draw [ultra thick] (2-6,0) -- (2-6,3);
             \draw [ultra thick] (2.5-6,0) -- (2.5-6,3); 
           \draw [<->] (0.5-6,-0.3) -- (1-6, -0.3);  
           \node at (0.75-6, -0.6) {$1/m$};
                       \draw (0,0) -- (0,3);
            \filldraw (0, 0.5) circle (0.08cm);
            \filldraw (0, 2.5) circle (0.08cm);
           \filldraw (0, 1.5) circle (0.08cm);   
             \draw (0.86,0) -- (0.86,3); 
                        \filldraw (0.86, 0) circle (0.08cm);
                                   \filldraw (0.86, 3) circle (0.08cm); 
           \filldraw (0.86, 1) circle (0.08cm); 
            \filldraw (0.86, 2) circle (0.08cm);   
           \draw (1.73,0) -- (1.73,3);   
               \filldraw (1.73, 0.5) circle (0.08cm);
            \filldraw (1.73, 2.5) circle (0.08cm);
           \filldraw (1.73, 1.5) circle (0.08cm);   
                 \filldraw (3*0.86, 0) circle (0.08cm);
            \filldraw (3*0.86, 3) circle (0.08cm);       
         \draw (1.73 + 0.86,0) -- (1.73 + 0.86,3);   
               \filldraw (1.73+0.86, 1) circle (0.08cm); 
            \filldraw (1.73+0.86, 2) circle (0.08cm);
                 \draw [<->] (0,-0.3) -- (0.8, -0.3);  
           \node at (0.4, -0.6) {$1/m$};
           \draw [<->] (-0.4, 0.5) -- (-0.4, 1.5);
           \node at (-0.8, 1) {\LARGE $\frac{5}{4m}$};
           \end{tikzpicture}
        \caption{$m$ closed loops of length 1 each containing $4m/5$ points that are equispaced and have an alternating offset.}
        \label{fig:config}
    \end{figure}
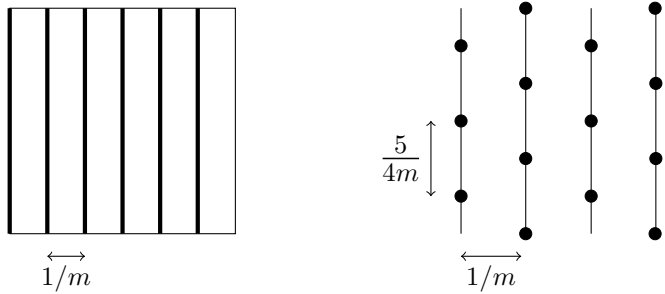    
    \end{center}

The Riesz $p-$energy of this initial configuration can be bounded from above: note that the symmetry of the configuration implies that
$$\sum_{i,j=1 \atop i \neq j}^{n} \frac{1}{\|x_i - x_j\|^p} = n \cdot \sum_{j=2}^{n} \frac{1}{\|x_1 - x_j\|^p}.$$
The nearest distance between any pair of points will be denoted by 
$$ \delta = \min_{i \neq j} \|x_i - x_j\|.$$
It is attained by adjacent points on neighboring lines since
$$ \delta = \sqrt{ \frac{1}{m^2} + \left(\frac{5}{8m}\right)^2 } = \frac{\sqrt{89/64}}{m} \sim \frac{1.179}{m} < \frac{5}{4m}.$$
This means that, for any given point, there are 4 points at distance $\delta$, two points at distance $\geq 1.06\delta$ and
all other points at least distance $1.69\delta$. Therefore
$$\sum_{i,j=1 \atop i \neq j}^{n} \frac{1}{\|x_i - x_j\|^p} \leq \frac{6 n}{\delta^p}+ \frac{n^2}{(1.69\delta)^p}.$$
More refined estimates would be possible, however, the argument is bound to lose a logarithm in another step which makes it unnecessary to further optimize this step.
 Now we modify the initial configuration by selecting a subset of $2m/5$ out of the $4m/5$ horizontal lines corresponding to an even index line, and removing the associated points. This corresponds to deleting $(2m/5) \cdot (m/2) = m^2/5$ points. We are left with a set containing exactly $3m^2/5$ points.

   \begin{center}
   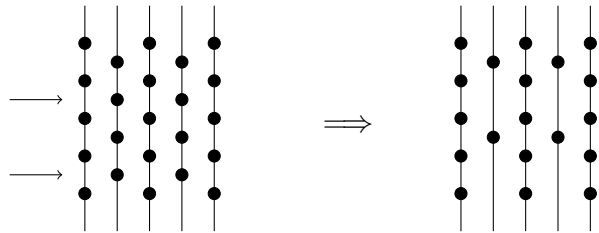
\begin{figure}[h!]
        \begin{tikzpicture}
                       \draw (0,0) -- (0,3);
            \filldraw (0, 0.5) circle (0.08cm);
                     \filldraw (0, 1) circle (0.08cm);
        \filldraw (0, 2) circle (0.08cm);
            \filldraw (0, 2.5) circle (0.08cm);
           \filldraw (0, 1.5) circle (0.08cm);   
             \draw (0.86/2,0) -- (0.86/2,3); 
         \filldraw (0.86/2, 0.75) circle (0.08cm);       
           \filldraw (0.86/2, 1.25) circle (0.08cm); 
             \filldraw (0.86/2, 1.75) circle (0.08cm); 
            \filldraw (0.86/2, 2.25) circle (0.08cm);   
                       \draw (0.86,0) -- (0.86,3);
            \filldraw (0.86, 0.5) circle (0.08cm);
                     \filldraw (0.86, 1) circle (0.08cm);
        \filldraw (0.86, 2) circle (0.08cm);
            \filldraw (0.86, 2.5) circle (0.08cm);
           \filldraw (0.86, 1.5) circle (0.08cm);  
             \draw (3*0.86/2,0) -- (3*0.86/2,3); 
         \filldraw (3*0.86/2, 0.75) circle (0.08cm);       
           \filldraw (3*0.86/2, 1.25) circle (0.08cm); 
             \filldraw (3*0.86/2, 1.75) circle (0.08cm); 
            \filldraw (3*0.86/2, 2.25) circle (0.08cm);   
                          \draw (2*0.86,0) -- (2*0.86,3);
            \filldraw (2*0.86, 0.5) circle (0.08cm);
                     \filldraw (2*0.86, 1) circle (0.08cm);
        \filldraw (2*0.86, 2) circle (0.08cm);
            \filldraw (2*0.86, 2.5) circle (0.08cm);
           \filldraw (2*0.86, 1.5) circle (0.08cm);  
           \draw[->] (-1, 0.75) -- (-0.3, 0.75);
              \draw[->] (-1, 1.75) -- (-0.3, 1.75);  
              \node at (3.5, 1.4) {\Large $\implies$};
                       \draw (0+5,0) -- (0+5,3);
            \filldraw (0+5, 0.5) circle (0.08cm);
                     \filldraw (0+5, 1) circle (0.08cm);
        \filldraw (0+5, 2) circle (0.08cm);
            \filldraw (0+5, 2.5) circle (0.08cm);
           \filldraw (0+5, 1.5) circle (0.08cm);   
             \draw (0.86/2+5,0) -- (0.86/2+5,3); 
           \filldraw (0.86/2+5, 1.25) circle (0.08cm); 
            \filldraw (0.86/2+5, 2.25) circle (0.08cm);   
                       \draw (0.86+5,0) -- (0.86+5,3);
            \filldraw (0.86+5, 0.5) circle (0.08cm);
                     \filldraw (0.86+5, 1) circle (0.08cm);
        \filldraw (0.86+5, 2) circle (0.08cm);
            \filldraw (0.86+5, 2.5) circle (0.08cm);
           \filldraw (0.86+5, 1.5) circle (0.08cm);  
             \draw (3*0.86/2+5,0) -- (3*0.86/2+5,3); 
           \filldraw (3*0.86/2+5, 1.25) circle (0.08cm); 
            \filldraw (3*0.86/2+5, 2.25) circle (0.08cm);   
                          \draw (2*0.86+5,0) -- (2*0.86+5,3);
            \filldraw (2*0.86+5, 0.5) circle (0.08cm);
                     \filldraw (2*0.86+5, 1) circle (0.08cm);
        \filldraw (2*0.86+5, 2) circle (0.08cm);
            \filldraw (2*0.86+5, 2.5) circle (0.08cm);
           \filldraw (2*0.86+5, 1.5) circle (0.08cm);           
           \end{tikzpicture}
           \caption{Deleting all the points corresponding to a horizontal line with even index.}
           \label{fig:setup}
    \end{figure}    
    \end{center}

 The number of initial configurations that can be so obtained is
 $$ \binom{\frac{4m}{5}}{\frac{2m}{5}} \geq 2^{m/2}.$$
 These configurations are all distinct.  However, if we consider global shifts in the $y-$direction, meaning taking a set of points $\mathcal{P}  \subset \mathbb{T}^2$ and sending
 $$ \mathcal{P}  \rightarrow \left\{ (p_{i,1}, p_{i,2} + y \mod 1): p_i \in \mathcal{P} \right\},$$
 then some of these configuration are no longer distinguishable. It is easy to see each initial configuration can give rise to no more than $m$ other configurations using the global shift and thus there are at least
  $$ \frac{1}{m}\binom{\frac{4m}{5}}{\frac{2m}{5}} \geq \frac{2^{m/2}}{m}$$
many initial configurations that are distinguishable even if we are only given them up to a global shift in the $y-$direction. Since removing points can only decrease the energy, we inherit the uniform bound from above and can conclude that for any such initial configuration
$$\sum_{i \neq j}^{} \frac{1}{\|x_i - x_j\|^p} \leq \frac{6 n}{\delta^p} + \frac{n^2}{(1.69\delta)^p}.$$

\textbf{Step 2.} 
For any such initial configuration, we consider the gradient flow induced by the energy $\sum_{i \neq j} 1/\|x_i - x_j\|^p$.  The first observation is that, by symmetry of the initial configuration, the gradient flow is highly restricted: each point stays on the line that it starts out on.  One way of seeing this is to write the gradient descent
$$ -\nabla E\big|_{x=x_i} = \sum_{j=1}^{n} \frac{x_i - x_j}{\|x_i - x_j\|^{p+2}}$$
and then note that for each point exerting a force on $x_i$ that is not on the same line, there is a symmetrically arranged point on the other side of the line whose force is reflected around the line (note that these forces need not cancel but, in the orientation of Fig. \ref{fig:setup}, their $x-$coordinates cancel leaving the points constrained to their respective lines). Note, moreover, that the initial configuration $P = \left\{x_1, \dots, x_n\right\} \subset \mathbb{T}^2$ is invariant under shifts in the sense that
$$ x_i \in P \implies x_i \pm \left(\frac{2}{m}, 0\right) \in P.$$
This means that the first, the third and, more generally, all the odd-numbered lines are all going to evolve identically and so will the second, the fourth and all even-numbered lines. Phrased differently, the entire dynamics is captured by the first two lines, the rest evolves identically.
The energy of the initial configuration is bounded, the energy is decreasing along the gradient flow. Since the energy is bounded from below and the configuration space is compact, the configuration will approach a critical point of the energy functional as time approaches infinity.\\

\textbf{Step 3.} It remains to show that we can recover the initial configuration from the critical point. We have already seen the entire evolution is shown in the first two lines, the remaining lines simply repeat periodically. We will only keep track of the first two lines and refer to points on the first line as red points and points on the second line as blue points. With that terminology, we might even go so far as to consider them both as lying on a single line (and interacting via a more complicated functional depending on the colors of the points), see Fig. \ref{fig:color}.

    \begin{center}
   \begin{figure}[h!]
        \begin{tikzpicture}
                       \draw (0+5,0) -- (0+5,3);
         \filldraw[red] (0+5, 0.5) circle (0.08cm);
         \filldraw[red]  (0+5, 1) circle (0.08cm);
        \filldraw[red]  (0+5, 2) circle (0.08cm);
            \filldraw[red]  (0+5, 2.5) circle (0.08cm);
           \filldraw[red]  (0+5, 1.5) circle (0.08cm);   
             \draw (0.86/2+5,0) -- (0.86/2+5,3); 
           \filldraw[blue] (0.86/2+5, 1.25) circle (0.08cm); 
            \filldraw[blue] (0.86/2+5, 2.25) circle (0.08cm);    
            \node at (6.75, 1.5) {$\implies$};
             \draw (8,0) -- (8,3);  
           \filldraw[red] (8, 0.5) circle (0.08cm);
         \filldraw[red]  (8, 1) circle (0.08cm);
        \filldraw[red]  (8, 2) circle (0.08cm);   
                    \filldraw[red]  (8, 2.5) circle (0.08cm);
           \filldraw[red]  (8, 1.5) circle (0.08cm);   
           \filldraw[blue] (8, 1.25) circle (0.08cm); 
            \filldraw[blue] (8, 2.25) circle (0.08cm);    
           \end{tikzpicture}
           \caption{First two rows of an initial configuration and a projection.}
           \label{fig:color}
    \end{figure}
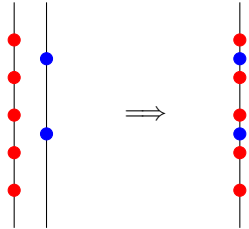    
    \end{center}

The main idea is now as follows: if $p$ is sufficiently large, then the projected evolution of the red and blue points is non-crossing. A (projected) red point and a (projected) blue point can never occur at the same position in space (i.e. on the projected line). Their relative ordering stays preserved (up to possibly jointly moving up or down in the $y-$direction). This is seen as follows: we have $n = 3m^2/5$ points in total. If projections of a red and blue point would occupy the same position, then this would mean that we have $m$ points (in $\mathbb{T}^2$) on a line. This contributes at least $m$ large terms to the energy and
$$ E \geq \frac{m}{(1/m)^p}.$$
We recall that the minimal distance in the initial configuration satisfied
$$ \delta = \sqrt{ \frac{1}{m^2} + \left(\frac{5}{8m}\right)^2 } = \frac{\sqrt{89/64}}{m} \sim \frac{1.179}{m}.$$
We can rewrite the lower bound on the energy of the critical point as
$$ E \geq \frac{m}{(1/m)^p} = \frac{m}{\delta^p}\cdot 1.179^p$$
Combining this with $m \geq \sqrt{n}$ and the previously derived upper bound on the initial configuration, we end up at
\begin{align*}
    \frac{\sqrt{n}}{\delta^p}\cdot 1.179^p \leq E(\mbox{critical}) \leq E(\mbox{initial}) \leq \frac{6 n}{\delta^p} + \frac{n^2}{(1.69\delta)^p}.
\end{align*}
This implies
$$   1.179^p \leq   6 \sqrt{n} +  \frac{n^{3/2}}{1.69^p}. $$
It is easy to see that if $1.179^p \geq 8 \sqrt{n}$, then $1.69^p \geq (1.179)^{3p} \geq 512 n^{3/2}$ and we get a contradiction. This contradiction arises as soon as 
$$ p \geq \frac{\log(8 \sqrt{n})}{\log(1.179)} \sim 3.03\log{n} + 12.62$$
which implies the desired result.
\end{proof}

\textbf{Remark.} The role of the condition $p \geq 5 \log{n}$ is clear: it ensures that no crossing of red and blue points can occur and this enables us to recover the initial configuration from the critical point. However, it appears that this condition can probably be relaxed (at least for a large class of initial configurations). We illustrate this with an example shown in Fig. \ref{fig:ill}. 
 Fig \ref{fig:ill} (left) shows the type of initial configuration that we consider in the proof: 3 lines have been removed. We can then apply gradient descent with respect to $p=15$ and end up in the critical point shown in the middle. The proof shows that (up to a global rotation in the $y-$direction) we can recover the initial configuration from the critical point because no crossing of points can ever take place, their relative order is preserved. If we now compare this to the critical point obtained when considering the logarithmic energy, we see that it also appears to be quite similar to the other two configurations. The non-crossing (presumably) also cannot take place for this particular initial configuration, however, this is not covered by the argument in its present form, some new ideas are required.

\begin{center}
    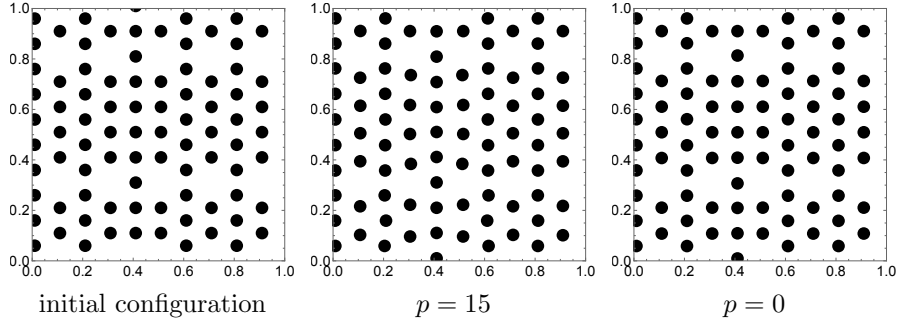
\begin{figure}[h!]
\begin{tikzpicture}
    \node at (0,0) {\includegraphics[width=0.3\textwidth]{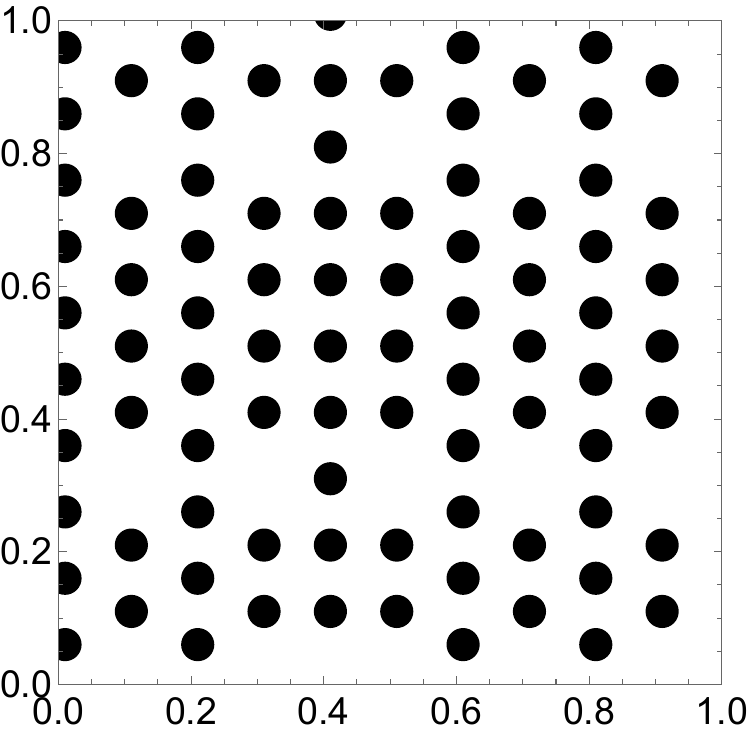}};
    \node at (0, -2.2) {initial configuration};
        \node at (4,0) {\includegraphics[width=0.3\textwidth]{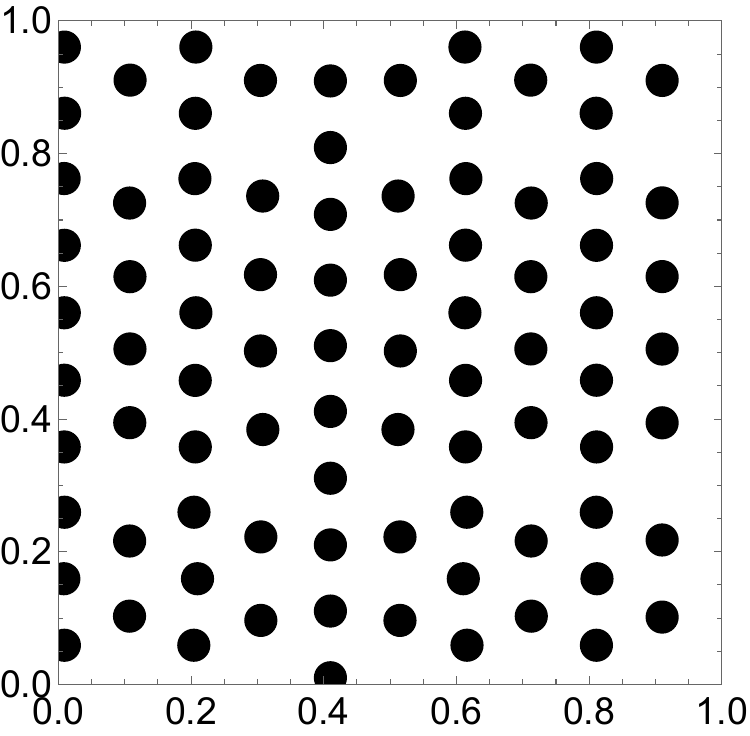}};
          \node at (4, -2.2) {$p=15$};
    \node at (8,0) {\includegraphics[width=0.3\textwidth]{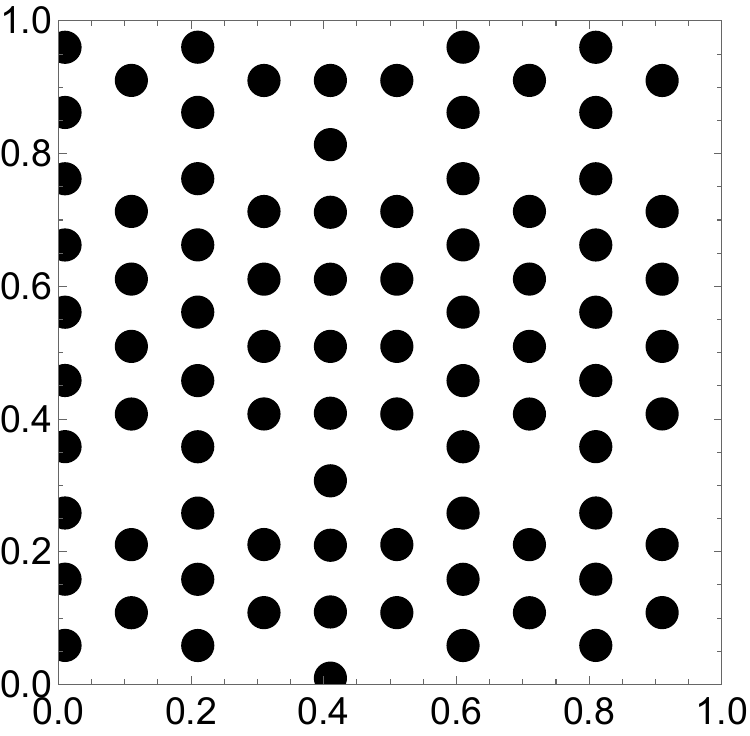}};
      \node at (8, -2.2) {$p=0$};
\end{tikzpicture}
\vspace{-10pt}
\caption{A configuration and critical points reached by gradient descent for $p=15$ (middle) and the logarithmic energy (right). }
\label{fig:ill}
    \end{figure}
\end{center}

\section{Proof of Theorem 2 and Theorem 3}
The proof of Theorem 3 can be used to prove Theorem 2, we therefore prevent the arguments in reverse order and start with $n=3$ points.

\subsection{Proof of Theorem 3, $n=4$}
\begin{proof} We use the one-parameter family $S_{\alpha}$ introduced above
$$ S_{\alpha} = \left\{ (0,0), \left( \tfrac{\tan{\alpha}}{2} , \tfrac12 \right), \left( \tfrac{1}{2}, \tfrac{ \tan{\alpha}}{2} \right), \left( \tfrac{1 + \tan{\alpha}}{2} , \tfrac{1 + \tan{\alpha}}{2} \right) \right\}$$
and the fact that for all $0 \leq \alpha \leq \pi/12$, the energy of $S_{\alpha} = \left\{s_1, s_2, s_3, s_4 \right\}$ is 
$$ \sum_{i,j=1 \atop i < j}^{4} \frac{1}{\|s_i - s_j\|^p} =4 \left( \frac{2}{\sec{\alpha}} \right)^p + 2 \left(\frac{\sqrt{2}}{1 - \tan{\alpha}}\right)^p.$$
The minimal distance between two elements in $S_{\pi/12}$ is 
$$ \min_{x,y \in S_{\pi/12} \atop x \neq y} \|x-y\| = \frac{\sqrt{6}-\sqrt{2}}{2}$$
which means that $S_{\pi/12} = \left\{s_1, s_2, s_3, s_4 \right\}$ has energy
$$ \sum_{i,j=1 \atop i < j}^{4} \frac{1}{\|s_i - s_j\|^p}  = 6 \cdot \left( \frac{2}{\sqrt{6}-\sqrt{2}}\right)^p.$$
A lengthy but straightforward computation shows that when computing the derivative of the energy $E(S_{\alpha})$ as a function of $\alpha$ at the point $\alpha = \pi/12$, then
\begin{align*}
 \frac{\partial}{\partial \alpha}  \sum_{i,j=1 \atop i < j}^{4} \frac{1}{\|s_i - s_j\|^p}  \quad \big|_{\alpha = \tfrac{\pi}{12}} &=   \frac{\sqrt{3} p \cdot 2^{1 + p/2}}{(\sqrt{3} - 1)^{p-2}}  > 0.
 \end{align*}
This means that the energy of $S_{\pi/12 - \varepsilon}$ is smaller than the energy of $S_{\pi/12}$ for all $\varepsilon >0$ sufficiently small (depending on $p$). Therefore $S_{\pi/12}$ is never a minimizer. Pick now an arbitrary value of $p$ and consider $x_1, x_2, x_3, x_4 \in \mathbb{T}^2$ to be any global minimizer of the $p-$Riesz energy. By the previous reasoning, this set has to be distinct from $S_{\pi/12}$. Since the set is distinct from $S_{\pi/12}$ and $S_{\pi/12}$ is the only set maximizing the minimal distance between any pair of points \cite[Proposition 4.3]{dick}, we have
 $$ \min_{i \neq j} \|x_i- x_j\| < \frac{\sqrt{6}-\sqrt{2}}{2}.$$
which implies that this particular configuration has $p-$energy at least
$$ \sum_{i,j=1 \atop i < j}^{4} \frac{1}{\|s_i - s_j\|^p} \geq \left( \frac{2}{\sqrt{6}-\sqrt{2}} + \eta\right)^p $$
for some $\eta > 0$. For $p$ sufficiently large, this lower bound is larger than the energy of $S_{\pi/12}$. As a consequence, each energy configuration can only be optimal for a finite range of values $p$ and, as $p$ varies from $0$ to $\infty$, there have to be infinitely many distinct minimizer.
\end{proof}

\subsection{Proof of Theorem 2, $n=3$}
\begin{proof} The case $n=3$ is almost identical to the argument for $n=4$. We use a very similar one-parameter family of points by taking $S_{\alpha}$ and erasing the last point
$$ T_{\alpha} = \left\{ (0,0), \left( \tfrac{\tan{\alpha}}{2} , \tfrac12 \right), \left( \tfrac{1}{2}, \tfrac{ \tan{\alpha}}{2} \right) \right\}$$
which we appreciate as $\left\{t_1, t_2, t_3\right\}$. A computation shows
$$ \frac{1}{\|t_1 - t_2\|^p} +  \frac{1}{\|t_1 - t_3\|^p} +  \frac{1}{\|t_2 - t_3\|^p} =  2\left( \frac{2}{\sec{\alpha}} \right)^p + \left(\frac{\sqrt{2}}{1 - \tan{\alpha}}\right)^p$$
which is exactly half the energy of the configuration $S_{\alpha}$ from the previous proof. 
In particular, the derivative at $\alpha = \pi/12$ is thus half the derivative of above and still positive
\begin{align*}
 \frac{\partial}{\partial \alpha}  \sum_{i,j=1 \atop i < j}^{3} \frac{1}{\|t_i - t_j\|^p}  \quad \big|_{\alpha = \tfrac{\pi}{12}} &=   \frac{\sqrt{3} p \cdot 2^{1 + p/2}}{(\sqrt{3} - 1)^{p-2}} > 0.
 \end{align*}
 Just as before, this shows that $T_3=T_{\pi/12}$ is never a minimizer. However,
 $$ \min_{i \neq j} \|t_i - t_j\| = \frac{\sqrt{6}-\sqrt{2}}{2}$$
 and for any set of three points distinct from $T_{\pi/12}$, using \cite[Proposition 4.2]{dick},
 $$ \min_{i \neq j} \|x_i- x_j\| < \frac{\sqrt{6}-\sqrt{2}}{2}.$$
 The argument then proceeds as above: each global minimizer for a fixed $p$ has to be distinct from $T_{\pi/12}$ which forces it to have at least one distance smaller than $(\sqrt{6} - \sqrt{2})/2$. This shows that for $p$ sufficiently large that particular configuration has a larger energy than $T_{\pi/12}$ from which we can conclude that it cannot be a minimizer for $p$ sufficiently large. Therefore, as $p$ varies from 0 to $\infty$, there have to be infinitely many different global minimizers.
\end{proof}

\textbf{Remark.} The fact that the same argument works for both $n=3$ and $n=4$ is a little bit surprising. There is an interesting and somewhat curious reason: the problem of maximizing the minimal distance between two pairs of points has the same answer for $n=3$ and $n=4$. The reason is that the extremal configuration for $n=3$ disks leaves room for a fourth disk of the same size (see Fig. \ref{fig:sur}).

\begin{center}
    \begin{figure}[h!]
     \begin{tikzpicture}
         \draw (8,0) -- (10,0) -- (10,2) -- (8,2) -- (8,0);
      \filldraw (8,0) circle (0.06cm);
      \filldraw (9,0.266) circle (0.06cm);
    \filldraw (8.266, 1) circle (0.06cm);
     \filldraw (9.266, 1.266) circle (0.03cm);
     \draw [thick] (8,0) circle (0.51cm);
     \draw (10,2) circle (0.51cm);
          \draw  (8,2) circle (0.51cm);
            \draw (10,0) circle (0.51cm);   
     \draw [thick] (9,0.2660) circle (0.51cm); 
          \draw  (9,2.2660) circle (0.51cm); 
          \draw  (11,0.2660) circle (0.51cm); 
           \draw  (10.266, 1) circle (0.51cm); 
       \draw [thick] (8.266,1) circle (0.51cm);  
     \draw [dashed] (9.266, 1.2660) circle (0.51cm);     
    \node at (9, -0.5) {$T_{\pi/12}$};
     \end{tikzpicture}
        \caption{Points maximizing minimal distance between three points (thick), copies for clarity, leave room for a fourth (dashed).}
        \label{fig:sur}
    \end{figure}
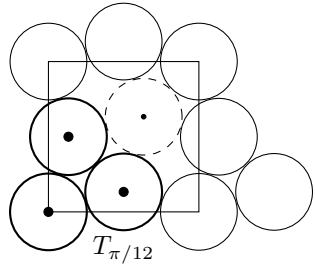
\end{center}

\section{Proof of Theorem 4}

\subsection{Preparatory statements}
The argument makes extensive use of the optimal ball packing result of Dickinson-Guillot-Keaton-Xhumari \cite{dick}. They show that the Fibonacci set on 5 points is the unique equal circle arrangement maximizing the density. This implies that for any $x_1, \dots, x_5 \in \mathbb{T}^2$
$$ \min_{i \neq j} \|x_i - x_j\| \leq \frac{1}{\sqrt{5}}$$
with equality if and only if the 5 points are the Fibonacci set
$$ p_1 = (0,0), \quad p_2 = (\tfrac15, \tfrac25), \quad p_3 = (\tfrac25, \tfrac45), \quad p_4 = (\tfrac35, \tfrac15), \quad p_5 = (\tfrac45, \tfrac35)$$
up to the usual symmetries. This implies that for any set of $5$ points
$$ \sum_{i,j=1 \atop i \neq j}^{5} \frac{1}{\|x_i - x_j\|^p}  \geq 5^{p/2}.$$
Moreover, the Fibonacci set only has distances $1/\sqrt{5}$ between its elements and thus
$$ \min_{x_1, \dots, x_5 \in \mathbb{T}^2} \quad \sum_{i,j=1 \atop i \neq j}^{5} \frac{1}{\|x_i - x_j\|^p} \leq 20 \cdot 5^{p/2}.$$
If the 5 points are not the Fibonacci set (up to symmetries), then at least one of the distances is smaller than $1/\sqrt{5}$ and thus there exists $\eta > 0$ such that
$$ \sum_{i,j=1 \atop i \neq j}^{5} \frac{1}{\|x_i - x_j\|^p}  \geq (5 + \eta)^{p/2},$$
where $\eta = \eta(x_1, \dots, x_5)$ depends on the points: this shows, just as above, that any configuration distinct from $F_5$ can only be optimal for a finite range of values of $p$. Moreover, for any $\varepsilon > 0$ there exists $p_0 > 0$ such that for all $p > p_0$ the optimal configuration of points has to satisfy
$$ \min_{i \neq j} \|x_i - x_j\|  \geq \frac{1}{\sqrt{5}} - \varepsilon.$$
Indeed, we can say a little bit more: if $x_1, \dots, x_n \in \mathbb{T}^5$ is a minimal energy configuration, then the comparison with the Fibonacci set $F_5$ implies
$$ 20 \cdot 5^{p/2} \geq \min_{i \neq j} \frac{1}{\|x_i - x_j\|^p}$$
which in turn requires, setting $\min_{x_i \neq x_j} \|x_i - x_j\| = 1/\sqrt{5} - \varepsilon$, that
$$ \frac{1}{\left( 1/\sqrt{5} - \varepsilon \right)^p} \leq 20\cdot 5^{p/2} \qquad \mbox{and thus} \quad (1 - \sqrt{5} \varepsilon)^p \geq 0.05.$$
Making the ansatz $\varepsilon = c/p$, we can argue
$$ 0.05 \leq \left(1 - \frac{\sqrt{5}c}{p}\right)^p \leq e^{-\sqrt{5} c}$$
implying that $c \geq 1$. We deduce that any energy minimizer has to satisfy
$$ \min_{i \neq j} \|x_i - x_j\|  \geq \frac{1}{\sqrt{5}} - \frac{1}{p}.$$

Compactness of the space and uniqueness of the Fibonacci set (proven in \cite{dick}) implies that it is enough to study small perturbations of the Fibonacci set $F_5$. So far, the argument is very much aligned with the behavior for $n=3,4$: when $p$ is large, any optimal configuration has to be very close to the solution of the disk packing problem. In the cases $n=3,4$ there exists an infinitesimal perturbation of the points that decreases the energy. 
It remains to show that this cannot happen for the Fibonacci set $F_5$. That argument comes in two parts (covered by the next two sections): suppose we are given $F_5$ and move each point a little bit, assume the point that is being moved the furthest is moved distance $\varepsilon \leq 1/1000$. Then
\begin{enumerate}
    \item the sum of the distances cannot change more than $\mathcal{O}(\varepsilon^2)$ [Lemma 1]
    \item but some distances change at least by $c \cdot \varepsilon$. [Lemma 2]
\end{enumerate}
This means that we go from having the same distance $1/\sqrt{5}$ twenty times to twenty distances that still sum, up to second order terms, to $20/\sqrt{5}$ with some of them being order $\pm \varepsilon$ away from $1/\sqrt{5}$. At this point, we recall Jensen's inequality and the convexity of $x \rightarrow 1/x^p$ to deduce that
$$ \frac{1}{20}\sum_{i \neq j} \frac{1}{\|x_i - x_j\|^p} \geq \frac{1}{ \left( \frac{1}{20} \sum_{i \neq j} \|x_i - x_j\|  \right)^p }.$$
Lemma 1 implies that 
$$ \frac{1}{20} \sum_{i \neq j} \|x_i - x_j\| = \frac{1}{\sqrt{5}} + \mathcal{O}(\varepsilon^2)$$
which is \textit{almost} what we want: however, it is conceivable that the error term is positive and the lower bound obtained this way would be strictly smaller than what we try to prove (which is a lower bound of $20 \cdot 5^{p/2}$).  At this point we invoke Lemma 2 to argue that the $\|x_i - x_j\|$ are not all too close to $1/\sqrt{5}$ which implies that we expect an order $\mathcal{O}(\varepsilon^2)$ loss if we apply Jensen's inequality. The goal is then to recover that loss; this part of the argument takes places at order $\mathcal{O}(\varepsilon^2)$ with two different factors contributing. Luckily, the implicit constant in front of the favorable term scales like $\sim p^2$ while the other term only scales as $\sim p$ which then implies the result for $p$ sufficiently large.

\subsection{A Sum of Distances Lemma.}

\begin{lemma}
    Let $F_5 = \left\{p_1, \dots, p_5\right\}$ and let $y_1, \dots, y_5 \in \mathbb{T}^2$ be five small points, $\max_{1 \leq i \leq 5} \|y_i\| \leq 1/1000$. There exists a constant $c_1$ so that
    $$ \left| \sum_{i,j=1 \atop i \neq j}^{5} \| (p_i + y_i) - (p_j + y_j) \| - \frac{20}{\sqrt{5}} \right| \leq  c_1 \sum_{i=1}^{5}\|y_i\|^2.$$
\end{lemma}
\begin{proof}
Since $\| p_i - p_j\| = 1/\sqrt{5}$ and we are only seeing small perturbations, the entire sum is analytic in the $y_i$. The inequality is clearly true when all the perturbations vanish, it is therefore equivalent to the vanishing of the first derivative with respect to $y_1, \dots, y_5$. Phrased differently, there is a power series expansion
$$  \sum_{i,j=1 \atop i \neq j}^{5} \| (p_i + y_i) - (p_j + y_j) \| = \frac{20}{\sqrt{5}} + \sum_{i=1}^{5} (a_i y_{i,1} + b_i y_{i,2}) + \mathcal{O}\left( \sum_{i=1}^{5}\|y_i\|^2\right)$$
and it suffices to show that the linear terms all vanish, i.e. that $a_i = 0 = b_i$. We note that this a statement for linear functions, it therefore suffices to prove it for a single $y_i \neq 0$ assuming that all other perturbations are 0. At this point, it is helpful to visualize the `nearest-neighbor-geometry' of the Fibonacci set which is best achieved by drawing the set and shifts.
\begin{center}
\begin{figure}[h!]
    \includegraphics[width=0.4\textwidth]{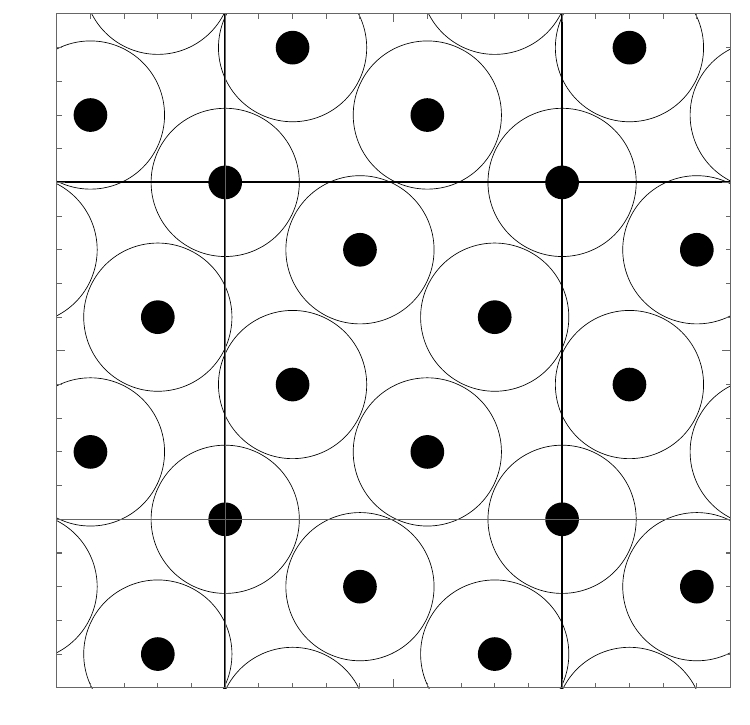}
    \caption{`Unwrapping' the Fibonacci set.}
    \label{fig:unwrap}
\end{figure}    
\end{center}
There is an underlying lattice structure. After rotation, we can assume that the point that we are perturbing lies in $(0,0)$, the four neighboring points are
$$ \left( \tfrac{1}{\sqrt{5}}, 0 \right),  \left(- \tfrac{1}{\sqrt{5}}, 0 \right),  \left( 0, \tfrac{1}{\sqrt{5}} \right) \quad \mbox{and} \quad  \left( 0, -\tfrac{1}{\sqrt{5}} \right).$$
The distance between $(0,0)$ to any of these points is $1/\sqrt{5} \ll 0.5$, which shows that the toroidal structure does not come into play, locally everything behaves in $\mathbb{R}^2$.
We define, for $x,y \in \mathbb{R}$ small,
\begin{align*}
    F(x,y) &= \left\| (x,y) - \left( \tfrac{1}{\sqrt{5}}, 0 \right)\right\| +  \left\| (x,y) - \left( -\tfrac{1}{\sqrt{5}}, 0 \right)\right\| \\
    &+  \left\| (x,y) - \left( 0, \tfrac{1}{\sqrt{5}} \right)\right\| +   \left\|(x,y) - \left(0, -\tfrac{1}{\sqrt{5}} \right)\right\|.
\end{align*} 
It is easy to see, either from the Pythagorean Theorem or an explicit computation, 
$$ \frac{\partial F}{\partial x}\big|_{(x,y) = (0,0)} = 0 \quad \mbox{and} \quad \frac{\partial F}{\partial y}\big|_{(x,y) = (0,0)} = 0.$$ 
This implies that the linear terms vanish and the claimed result follows.  A little bit more could be said: a more extensive computation shows
$$ F(x,y) = \frac{4}{\sqrt{5}} + \sqrt{5} (x^2 + y^2) + \mathcal{O}(|x|^3 + |y|^3).$$
which shows that the implicit constant is, infinitesimally, given by $\sqrt{5}$. If we want validity in a small neighborhood, say $\sqrt{x^2 + y^2} \leq 1/1000$, then the constant has to be increased a little bit (though, presumably, not very much); since this is not relevant for the rest of the argument, we do not further pursue this here (though, it might be relevant if one wanted to try to obtain an explicit bound on $p_0$).
\end{proof}

\subsection{A Linear Shift Lemma}
We will need a second Lemma that says the following: if we perturb all the points $p_i \rightarrow p_i + y_i$ then unless this is a global shift (meaning $y_i = y_j$ for all $i,j$), some of the distances change noticeably (meaning to first order). We note that since we are only interested in sums over differences, we can assume w.l.o.g. that the first point does not move, i.e. $y_1 = 0$.
\begin{lemma}
   There exists a universal constant $c>0$ such that if $F_5 = \left\{p_1, \dots, p_5\right\}$ denotes the Fibonacci points and $y_1, \dots, y_5 \in \mathbb{T}^2$ are five `small' points, such that $\max_{1 \leq i \leq 5} \|y_i\| \leq 1/1000$ and if $y_1 = (0,0)$, then
    $$ \min_{i \neq j} \| (p_i+ y_i) - (p_j + y_j) \| \leq \frac{1}{\sqrt{5}} - c \max_{1 \leq i \leq 5} \|y_i\|.$$
\end{lemma}
 Using Lemma 1, it actually suffices to show that at least one distance changes by at least $c \max_{1 \leq i \leq 5} \|y_i\|$ (either increasing or decreasing) because the sum is constant (up to a second order term). In particular, Lemma 2 (in conjunction with Lemma 1) also implies (and is basically equivalent to)
     $$ \max_{i \neq j} \| (p_i+ y_i) - (p_j + y_j) \| \geq \frac{1}{\sqrt{5}} + c \max_{1 \leq i \leq 5} \|y_i\|.$$
\begin{proof} Let us assume the statement is false: then, for any small $\eta > 0$, we can find a perturbation with the largest shift being of size $\varepsilon$ such that
$$ \forall k\neq \ell \qquad  \frac{1}{\sqrt{5}} - \eta \max_{1 \leq i \leq 5} \|y_i\| \leq \| (p_k+ y_k) - (p_{\ell} + y_{\ell}) \| \leq \frac{1}{\sqrt{5}} + \eta \max_{1 \leq i \leq 5} \|y_i\|.$$
\begin{center}
    \begin{figure}[h!]
    \begin{tikzpicture}[scale=2]
        \filldraw (0,0) circle (0.05cm);
        \node at (0.15, -0.15) {$p_1$};
        \filldraw (0.3,0.6) circle (0.05cm);
         \node at (0.3+0.12, 0.6-0.22) {$p_2$};
      \filldraw (-0.3,-0.6) circle (0.05cm);
        \node at (-0.3+0.15, -0.6-0.15) {$p_3$};
              \filldraw (0.6,-0.3) circle (0.05cm);
 \node at (0.6+0.15, -0.3-0.15) {$p_4$};
               \filldraw (-0.6,0.3) circle (0.05cm);
 \node at (-0.6+0.15, 0.3-0.15) {$p_5$};
                \draw [thick] (-0.6+1.5,0.3) circle (0.05cm);
 \node at (1.5-0.6+0.15, 0.3-0.15) {$p_5$};
 \draw (0,0) circle (0.68cm);
  \draw [thick] (0.3,0.6) circle (0.1cm);
  \draw [dashed] (1.5-0.6, 0.3) -- (0.3, 0.6);
           \node at (0.3+0.12-1.5, 0.6-0.22) {$p_2$};
      \draw [thick] (0.3-1.5,0.6) circle (0.05cm);
        \draw [dashed] (-0.6, 0.3) -- (0.3-1.5, 0.6);
    \end{tikzpicture}
    \caption{$p_1$ does not move; its nearest neighbors are shown. $p_2$ is being moved to one of the points on the smaller circle.}
    \label{fig:fib}
    \end{figure}
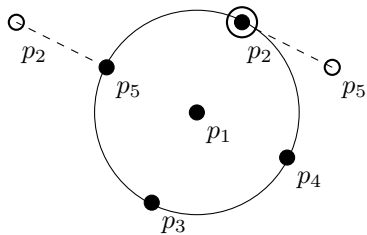
\end{center}

We now show that this will lead to a contradiction for $\eta$ sufficiently small.
 $p_1$ does not move, we think of it as in the center, surrounded by its four neighbors (see Fig. \ref{fig:fib}). At least one of the four neighbors of $p_1$ moves a distance of
$$\varepsilon := \max_{1 \leq i \leq 5} \|y_i\| \leq \frac{1}{1000}.$$
By symmetry, we can assume that the point that moves the most is $p_2$, i.e. $\|y_2\| = \varepsilon \geq \|y_i\|$. At this point we use that $p_2$ does not change its distance to $p_1$ by more than $\eta \varepsilon$ (with $0 < \eta \ll 1$ very small). This forces $y_2$ to be very nearly orthogonal to $p_2 - p_1$ (see Fig. \ref{fig:illu}).

\begin{center}
    \begin{figure}[h!]
    \begin{tikzpicture}[scale=1]
        \filldraw (0,0) circle (0.06cm);
        \node at (0, -0.3) {$p_1$};
                \filldraw (6,0) circle (0.06cm);
         \node at (6, -0.3) {$p_2$};
         \draw (6, 0) circle (1cm);
         \draw [<->] (7.5, 0) -- (7.5, 1);
         \node at (7.7, 0.5) {$\varepsilon$};
          \draw [thick,domain=0:15] plot ({6.2*cos(\x)}, {6.2*sin(\x)});
        \draw [thick,domain=345:360] plot ({6.2*cos(\x)}, {6.2*sin(\x)});
          \draw [thick,domain=0:15] plot ({5.8*cos(\x)}, {5.8*sin(\x)});
        \draw [thick,domain=345:360] plot ({5.8*cos(\x)}, {5.8*sin(\x)});     
        \draw [ultra thick,domain=84:105] plot ({6+cos(\x)}, {sin(\x)}); 
     \draw [ultra thick,domain=84:105] plot ({6+1.02*cos(\x)}, {1.02*sin(\x)});   
      \draw [ultra thick,domain=84:105] plot ({6+cos(\x)}, {-sin(\x)}); 
     \draw [ultra thick,domain=84:105] plot ({6+1.02*cos(\x)}, {-1.02*sin(\x)}); 
        \draw [<->] (5.6, -1.7) -- (6, -1.7);
        \node at (5.8, -2) {$\eta \varepsilon$};
    \end{tikzpicture}
    \caption{$p_2$ moves to $p_2 + y_2$ which is a point on the circle. Restrictions imposed by $\|p_1 - (p_2 + y_2)\|$ not changing more than $\eta \varepsilon$ forces the new point to be in the bold part.}
    \label{fig:illu}
    \end{figure}
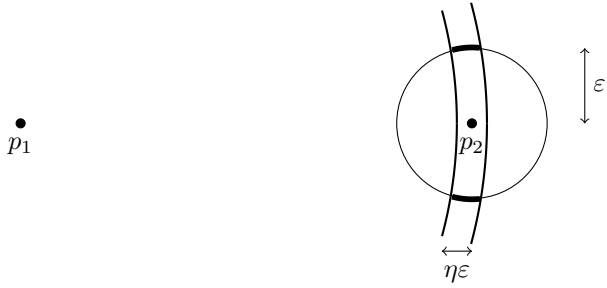
\end{center}

At this point, the global structure comes into play. We see that this implies that $p_2$ moving to $p_2 + y_2$ changes the distance it has to $p_5$ by at least a factor of $\varepsilon/2$
$$ \left| \| (p_2 + y_2) - p_5 \| - \| p_2 - p_5\| \right| \geq \frac{\varepsilon}{2}.$$
This, by itself, is not a contradiction yet. However, for $0 < \eta \ll 1$ very small, it forces $p_5$ to move into a direction $p_5 + y_5$ with 
$$ \left| \left\langle y_5 , \frac{y_2}{\|y_2\|} \right\rangle \right| \geq \frac{\varepsilon}{8}$$
since otherwise there is not enough length to compensate.  Then however, $p_5$ changes its distance to $p_1$ by at least $\varepsilon/20$ and we have a contradiction. 
\end{proof}

\subsection{Proof of Theorem \ref{prop:5}}
\begin{proof} We assume now that the 
$ F_5 = \left\{p_1, p_2, p_3, p_4, p_5 \right\} \subset \mathbb{T}^2$
is the Fibonacci set and that we are given a small perturbation $p_i \rightarrow p_i + y_i$ where  $$\varepsilon := \max_{1 \leq i \leq 5} \|y_i\| \leq 1/1000.$$
Since the entire problem only concerns the difference between points, it is invariant under global shifts of the entire set and we may assume without loss of generality that $y_1 = (0,0)$, i.e. the point $p_1$ does not move.
We introduce ten real variables $\delta_{ij}$ for $1 \leq i < j \leq 5$ via
$$  \| (p_i + y_i) - (p_j + y_j) \| = \frac{1}{\sqrt{5}} + \delta_{ij}.$$
Lemma 1 implies that for some constant $c_1 > 0$
$$ \left| \sum_{i, j = 1}^{5} \delta_{ij} \right| \leq c_1 \varepsilon^2$$
while Lemma 2 implies that for some constant $c_2 > 0$
$$ \min_{i,j} \delta_{ij} \leq - c_2 \varepsilon.$$
The goal is to show that
$$ \sum_{i, j =1 \atop i <j}^{5} \frac{1}{ \left( \frac{1}{\sqrt{5}} + \delta_{ij} \right)^p} \geq 10 \cdot 5^{p/2}$$
with equality if and only if all the $\delta_{ij} \equiv 0$. Without loss of generality and after possibly relabeling the points, we may assume that the smallest among the $\delta_{ij}$ is $\delta_{12}$. We use Jensen's inequality: for any $k$ positive real numbers $x_1, \dots, x_k$, the convexity of $x \rightarrow 1/x^p$ implies that
$$ \frac{1}{k}\sum_{\ell = 1}^{k} \frac{1}{x_{\ell}^p} \geq \left( \frac{1}{k} \sum_{\ell=1}^{k} x_{\ell} \right)^{-p} \qquad \mbox{or, equivalently} \qquad \sum_{\ell = 1}^{k} \frac{1}{x_{\ell}^p} \geq k\left( \frac{1}{k} \sum_{\ell=1}^{k} x_{\ell} \right)^{-p}.$$
We rewrite the sum by explicitly keeping the term involving $\delta_{12}$ and applying Jensen's inequality to the remaining $k=9$ terms. This leads to
\begin{align*}
    \sum_{i, j =1 \atop i <j}^{5} \frac{1}{ \left( \frac{1}{\sqrt{5}} + \delta_{ij} \right)^p} &\geq  \frac{1}{ \left( \frac{1}{\sqrt{5}} + \delta_{12} \right)^p} +  9 \left(\frac{1}{9}  \sum_{i, j =1 \atop i <j, (i,j) \neq (1,2)}^{5}  \left[ \frac{1}{\sqrt{5}} + \delta_{ij} \right] \right)^{-p} \\
    &= \frac{1}{ \left( \frac{1}{\sqrt{5}} + \delta_{12} \right)^p} +  9 \left(  \frac{1}{\sqrt{5}} + \frac{1}{9}  \sum_{i, j =1 \atop i <j, (i,j) \neq (1,2)}^{5}  \delta_{ij}  \right)^{-p}.
\end{align*}
At this point we use Lemma 1 to deduce that the sum over all changes $\delta_{ij}$ is small and therefore 
\begin{align*}
\frac{1}{9} \sum_{i, j =1 \atop i <j, (i,j) \neq (1,2)}^{5}  \delta_{ij} = -\frac{\delta_{12}}{9} + \frac{1}{9} \sum_{i, j =1 \atop i <j}^{5}  \delta_{ij} \leq  -\frac{\delta_{12}}{9} +  c_1 \varepsilon^2.
\end{align*}
This implies that 
\begin{align*}
    \sum_{i, j =1 \atop i <j}^{5} \frac{1}{ \left( \frac{1}{\sqrt{5}} + \delta_{ij} \right)^p} 
    &\geq  
     \frac{1}{ \left( \frac{1}{\sqrt{5}} + \delta_{12} \right)^p} +  9 \left(  \frac{1}{\sqrt{5}}   -\frac{\delta_{12}}{9} + c_1 \varepsilon^2 \right)^{-p}.
\end{align*}
Invoking Lemma 2 implies that 
$$\delta_{12} \leq - c_2 \varepsilon \implies \varepsilon^2 \leq \frac{\delta_{12}^2}{c_2^2}$$
and therefore
\begin{align*}
    \sum_{i, j =1 \atop i <j}^{5} \frac{1}{ \left( \frac{1}{\sqrt{5}} + \delta_{ij} \right)^p} 
    &\geq  
     \frac{1}{ \left( \frac{1}{\sqrt{5}} + \delta_{12} \right)^p} +  9 \left(  \frac{1}{\sqrt{5}}   -\frac{\delta_{12}}{9} + \frac{c_1}{c_2^2} \delta_{12}^2 \right)^{-p}.
\end{align*}
We abbreviate this new universal constant as $c= c_1/c_2^2 > 0$ and are led to the study of the function $h:[0,1/\sqrt{5}] \rightarrow \mathbb{R}$
$$ h(x) = \frac{1}{ \left( \frac{1}{\sqrt{5}} - x \right)^p} +  \frac{9}{\left(  \frac{1}{\sqrt{5}}   + \frac{x}{9} + c x^2 \right)^{p}}.$$
Of course $h(0) = 10 \cdot 5^{p/2}.$
It remains to show that for all $c>0$ there exists $p_c > 0$, depending only on $c$, such that for $p \geq p_c$, the function $h$ assumes its global minimum in 0. A short computation shows that $h'(0) = 0$
and
$$ h''(0) = \frac{2p}{9} 5^{p/2} \left( 25(1+p) - 81 \sqrt{5} c \right)$$
which is positive for $p$ sufficiently large ($p \geq 10c$ is enough). If we can now show that $h'(x)$ cannot have a positive root, monotonicity would follow: if $h'(0) = 0$ and $h''(0) > 0$, then $h'(x)$ is positive for all small enough positive $x$ and if the first derivative $h'$ cannot vanish, then it has to say positive, the function $h$ is monotonically increasing and its minimum is attained in 0. We have
\begin{align*}
 \frac{h'(x)}{p} &= \frac{1}{(1/\sqrt{5} - x)^{p+1}} - \frac{ 1 + 18 c x}{(1/\sqrt{5} + x/9 + c x^2)^{p+1}} \\
 &= \left[\frac{1}{(1/\sqrt{5} - x)^{p+1}} - \frac{1}{(1/\sqrt{5})^{p+1}}\right] + \left[\frac{1}{(1/\sqrt{5})^{p+1}}  - 
 \frac{ 1 + 18 c x}{(1/\sqrt{5} + x/9 + c x^2)^{p+1}} \right] \\
 &\geq  \left[\frac{1}{(1/\sqrt{5} - x)^{p+1}} - \frac{1}{(1/\sqrt{5})^{p+1}}\right] + \left[\frac{1}{(1/\sqrt{5})^{p+1}}  - 
 \frac{ 1 + 18 c x}{(1/\sqrt{5})^{p+1}} \right] \\
 &= \left[\frac{1}{(1/\sqrt{5} - x)^{p+1}} - \frac{1}{(1/\sqrt{5})^{p+1}}\right] - \frac{18 c x}{(1/\sqrt{5})^{p+1}}.
\end{align*}
At this point, we use the mean value theorem in the following form: for any differentiable $f:[a,b] \rightarrow \mathbb{R}$, we have
$$ f(b) - f(a) \geq \left( \min_{a \leq x \leq b} f'(x)\right) (b-a)$$
which we use to deduce that
 $$ \frac{1}{(1/\sqrt{5} - x)^{p+1}} - \frac{1}{(1/\sqrt{5})^{p+1}} \geq \frac{p+1}{(1/\sqrt{5})^{p+2}} x \geq \frac{p x}{(1/\sqrt{5})^{p+1}}  .$$
This dominates the other term for all $p > 18c$ and the result follows.
\end{proof}

\end{document}